\documentclass{birkjour_t2}
 \newtheorem{thm}{Theorem}[section]
 \newtheorem{lem}[thm]{Lemma}
 \newtheorem{prop}[thm]{Proposition}
 \theoremstyle{definition}
 \newtheorem{defn}[thm]{Definition}
 \theoremstyle{remark}
 \newtheorem{rem}[thm]{Remark}
 \numberwithin{equation}{section}

\usepackage[english]{babel}

\usepackage{epstopdf}

\usepackage{mathrsfs}

\usepackage{enumitem}

\usepackage{functan}
\Macro{C0T}{\mathscr{C}([0,T])}
\Macro{C10T}{\mathscr{C}^1([0,T])}
\Macro{Cw0TH01}{\mathscr{C}_w([0,T];\m{H01Om})}
\Macro{Cw0TL2}{\mathscr{C}_w([0,T];\m{L2Om})}
\Macro{Linfty0TH01}{\mathrm{L}^{\infty}(0,T;\m{H01Om})}
\Macro{Linfty0TL2}{\mathrm{L}^{\infty}(0,T;\m{L2Om})}
\Macro{CcinftyOm}{\mathscr{C}_c^{\infty}(\Omega)}
\Macro{Ccinfty0T}{\mathscr{C}_c^{\infty}(0,T)}

\Macro{CcinftyQ}{\mathscr{C}_c^{\infty}(Q)}
\Macro{CcinftyOmRd}{\mathscr{C}_c^{\infty}(\Omega;\R^d)}
\Macro{C0TL2}{\mathscr{C}([0,T];\m{L2Om})}
\Macro{C0TX}{\mathscr{C}([0,T];X)}
\Macro{AC0TX}{\mathscr{AC}([0,T];X)}
\Macro{Cw0TX}{\mathscr{C}_w([0,T];X)}
\Macro{AC0TL2}{\mathscr{AC}([0,T];\m{L2Om})}
\Macro{Lp}{\mathrm{L}^p}
\Macro{Lp0TLp}{\mathrm{L}^p(0,T;\m{LpOm})}
\Macro{L2}{\mathrm{L}^2}
\Macro{LinftyOm}{\mathrm{L}^{\infty}(\Omega)}
\Macro{L10T}{\mathrm{L}^1(0,T)}
\Macro{L1Om}{\mathrm{L}^1(\Omega)}
\Macro{L1OmRd}{\mathrm{L}^1(\Omega;\R^d)}
\Macro{L2Om}{\mathrm{L}^2(\Omega)}
\Macro{L10TL2}{\mathrm{L}^1(0,T;\m{L2Om})}
\Macro{LinftyOmRd}{\mathrm{L}^{\infty}(\Omega;\R^{d})}
\Macro{L20TH01}{\mathrm{L}^2(0,T;\m{H01Om})}
\Macro{L2OmRd}{\mathrm{L}^2(\Omega;\R^{d})}
\Macro{LpOm}{\mathrm{L}^p(\Omega)}
\Macro{LpOmRd}{\mathrm{L}^p(\Omega;\R^{d})}
\Macro{Lp0TX}{\mathrm{L}^p(0,T;X)}
\Macro{Lp0TLpOm}{\mathrm{L}^p(0,T;\m{LpOm})}
\Macro{LpQ}{\mathrm{L}^p(Q)}
\Macro{L10TX}{\mathrm{L}^1(0,T;X)}

\Macro{L2Q}{\mathrm{L}^2(Q)}
\Macro{L10Torspapot*OmRd}{\mathrm{L}^1(0,T;\m{orspapot*OmRd})}
\Macro{LinftyQRd}{\mathrm{L}^{\infty}(Q;\R^{d})}
\Macro{L10TH01}{\mathrm{L}^1(0,T;\m{H01Om})}
\Macro{L10TL2OmRd}{\mathrm{L}^1(0,T;\m{L2OmRd})}
\Macro{L1Q}{\mathrm{L}^1(Q)}
\Macro{L1QRd}{\mathrm{L}^1(Q;\R^{d})}
\Macro{L1QQk}{\mathrm{L}^1(Q\setminus Q_k)}
\Macro{Linfty0TX}{\mathrm{L}^{\infty}(0,T;X)}
\Macro{L20TL2}{\mathrm{L}^2(0,T;\m{L2Om})}
\Macro{H1Om}{\mathrm{H}^1(\Omega)}
\Macro{H01Om}{\mathrm{H}_0^1(\Omega)}
\Macro{W1inftyOm}{\mathrm{W}^{1,\infty}(\Omega)}
\Macro{Hr-1OmRd}{\mathrm{H}^{r-1}(\Omega;\R^{d})}
\Macro{HrH01}{\mathrm{H}^r(\Omega)\cap\mathrm{H}_0^1(\Omega)}
\Macro{Hr}{\mathcal{H}^r}
\Macro{W1pOm}{\mathrm{W}^{1,p}(\Omega)}
\Macro{W01pOm}{\mathrm{W}_0^{1,p}(\Omega)}
\Macro{W1p}{\mathrm{W}^{1,p}}
\Macro{W012Om}{\mathrm{W}_0^{1,2}(\Omega)}
\Macro{W12Om}{\mathrm{W}^{1,2}(\Omega)}
\Macro{HrOm}{\mathrm{H}^r(\Omega)}

\Macro{W110TL2}{\mathrm{W}^{1,1}(0,T;\m{L2Om})}
\Macro{W110TX}{\mathrm{W}^{1,1}(0,T;X)}
\Macro{W120TX}{\mathrm{W}^{1,2}(0,T;X)}
\Macro{W120TH1}{\mathrm{W}^{1,2}(0,T;\m{H1Om})}
\Macro{L20THr*}{\mathrm{L}^2(0,T;(\m{Hr})^*)}
\Macro{orclapotOmRd}{\mathscr{L}_{\pot}(\Omega;\R^d)}
\Macro{orspapotOmRd}{\mathrm{L}_{\pot}(\Omega;\R^d)}
\Macro{orEpotOmRd}{\mathrm{E}_{\pot}(\Omega;\R^{d})}
\Macro{orclapot*OmRd}{\mathscr{L}_{\pot^*}(\Omega,\R^{d})}
\Macro{orEpot*OmRd}{\mathrm{E}_{\pot^*}(\Omega;\R^d)}
\Macro{orspapot*OmRd}{\mathrm{L}_{\pot^*}(\Omega;\R^d)}

\Macro{orclapotQRd}{\mathscr{L}_{\pot}(Q;\R^{d})}
\Macro{orclapot*QRd}{\mathscr{L}_{\pot^*}(Q;\R^{d})}
\Macro{orspapotQRd}{\mathrm{L}_{\pot}(Q;\R^{d})}
\Macro{orspapot*QRd}{\mathrm{L}_{\pot^*}(Q;\R^{d})}
\Macro{orEpotQRd}{\mathrm{E}_{\pot}(Q;\R^{d})}
\Macro{orEpot*QRd}{\mathrm{E}_{\pot^*}(Q;\R^{d})}
\Macro{orspapot0TorspapotOmRd}{\mathrm{L}_{\pot}(0,T;\m{orspapotOmRd})}
\Macro{orspapotQ-introduction}{\mathrm{L}_{\pot}(Q)}
\Macro{orspapot0Torspapot-introduction}{\mathrm{L}_{\pot}(0,T;\mathrm{L}_{\pot}(\Omega))}
\Macro{2Om}{2,\Omega}
\Macro{potOm}{\pot,\Omega}
\Macro{pot*Om}{\pot^*,\Omega}
\Macro{potQ}{\pot,Q}

\newcommand{\nfun}{$\mathcal{N}$-function}
\newcommand{\eps}{\varepsilon}

\newcommand{\R}{\mathbb{R}}
\newcommand{\N}{\mathbb{N}}
\newcommand{\nonl}{\sigma}
\newcommand{\pot}{\phi}
\newcommand{\dt}{\operatorname{dt}}
\newcommand{\dx}{\operatorname{dx}}

\begin{document}

%
%
%
%
%
%
%
%
%

\title[Fully discrete scheme in elastodynamics in Orlicz spaces]{Convergence of a Full Discretization for a Second-Order\\ Nonlinear Elastodynamic Equation in Isotropic\\ and Anisotropic Orlicz Spaces}

\author{A. M. Ruf}

\address{%
Department of Mathematics\\
University of Oslo\\
P.O. Box 1053, Blindern\\
0316 Oslo\\
Norway}

\email{adrianru@math.uio.no}

\thanks{%
\begin{minipage}{1cm}
\includegraphics[height=1cm]{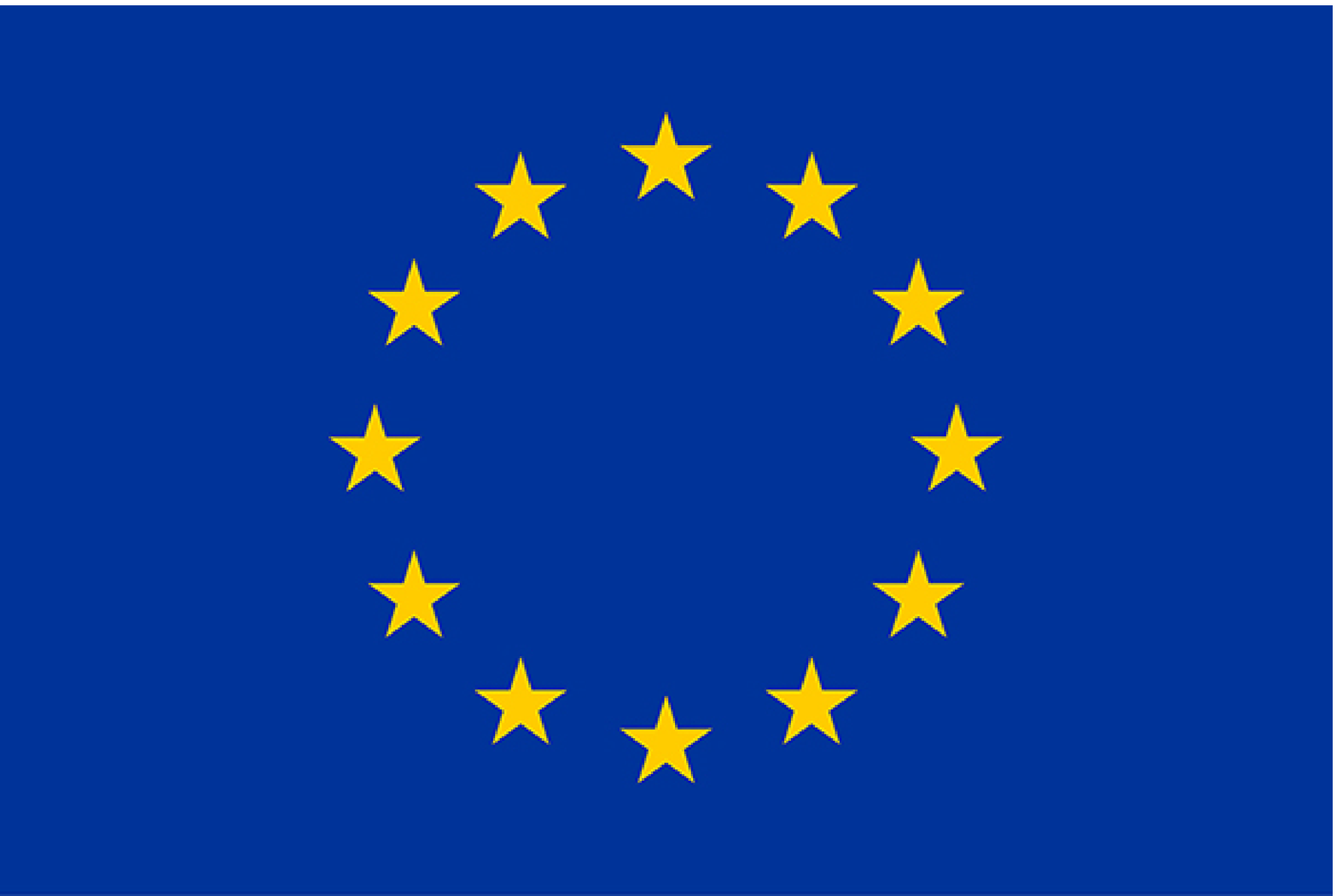}
\end{minipage}
\vspace{-1cm}\flushright This project has received funding from the European Union's Framework Programme for Research and\\ Innovation Horizon 2020 (2014-2020) under the Marie Sk{\l}odowska-Curie Grant Agreement No. 642768.
}


\date{\today}

\begin{abstract}
In this paper, we study a second-order, nonlinear evolution equation with damping arising in elastodynamics. The nonlinear term is monotone and possesses a convex potential but exhibits anisotropic and nonpolynomial growth.
The appropriate setting for such equations is that of monotone operators in Orlicz spaces. Global existence of solutions in the sense of distributions is shown via convergence of the backward Euler scheme combined with an internal approximation.
Moreover, we show uniqueness in a class of sufficiently smooth solutions and provide an a priori error estimate for the temporal semidiscretization.
\end{abstract}

\maketitle
\section{Introduction}
In this paper, we study the following second-order nonlinear hyperbolic elastodynamic equation
\begin{subequations}
\begin{align}
    	\partial_{tt}u- \Delta \partial_{t}u - \nabla \cdot \nonl(\nabla u) &= f &&\text{in }Q:=\Omega\times(0,T) \\
    	u&=0&&\text{on }\partial\Omega\times(0,T)\\
    	u(\cdot,0)=u_0,\, \partial_t u(\cdot,0)&=v_0 &&\text{in }\Omega.
\end{align}\label{prob1}%
\end{subequations}%
We want to show existence of solutions $u:\overline{\Omega}\times [0,T]\to\R$, where $\Omega\subset\R^d$ is a bounded Lipschitz domain, and $T>0$.
The stress $\nonl:\R^d\to\R^d$ is assumed to have the potential $\pot:\R^{d}\to \R$.
Such an equation arises (although as a system) in solid mechanics describing viscoelastic material \cite{klouvcek1994computation,klouvcek1994computational,rybka1992dynamical}.

In this paper, we will assume that the potential $\pot$ is an \nfun~(see definition \ref{def nfun}) and is thus convex. Note that the nonlinearity then is monotone such that
\begin{equation*}
	(\nonl(\xi)-\nonl(\eta))\cdot(\xi-\eta) \geq 0\quad \text{for all } \xi,\eta\in\R^d.
\end{equation*}
Moreover, we assume that $\nonl$ satisfies the following growth condition in terms of the \nfun~$\pot$ and its conjugate:
\begin{equation}
	\pot^*(\nonl(\xi))\leq C(1+\pot(\xi))\quad \text{for all }\xi\in\R^d. \label{growth condition}
\end{equation}
This constitutes a generalization of the current results concerning equation \eqref{prob1} since we do not need to assume polynomial growth of $\pot$.


Previous results regarding equation \eqref{prob1} in higher dimensions were obtained by Gajewski, Gr\"oger, and Zacharias \cite{zacharias1974nichtlineare}, Clements \cite{clements1974existence}, Friedman and Ne\u{c}as \cite{friedman1988systems}, Engler \cite{engler1989global}, and Rybka \cite{rybka1992dynamical}, all of which rely either on monotonicity or global Lipschitz continuity of $\nonl$.
More recent contributions from Friesecke and Dolzmann \cite{friesecke1997implicit} and Emmrich and \v{S}i\v{s}ka \cite{emmrich2013evolution} have generalized these results in the sense that they only assume that $\nonl${} satisfies the Andrews-Ball condition, i.e.  there exists $\lambda>0$ such that
\begin{equation*}
	(\nonl(\xi)-\nonl(\eta))\cdot (\xi-\eta)\geq -\lambda |\xi-\eta|^2\quad \text{for all }\xi,\eta\in\R^d,
\end{equation*}
which is fulfilled if $\nonl$ is monotone or globally Lipschitz continuous.

On the other hand, all of the contributions listed above critically rely upon $\nonl$ and its potential $\pot$ satisfying a polynomial growth and coercivity condition, i.e. there exists $p\geq 2$ and constants $C_1,C_2,C_3\geq 0$ such that
\begin{align*}
 	\pot(\xi) &\geq C_1|\xi|^p - C_2\\
 	\text{and }|\nonl(\xi)|&\leq C_3 (1+|\xi|)^{p-1}
 \end{align*}
which, in essence, is the growth condition \eqref{growth condition} for $\pot\sim |\cdot|^p$ and $\pot^*=C|\cdot|^{\frac{p}{p-1}}$.

However, in this paper we want to generalize the growth condition and allow for anisotropic and nonpolynomial growth. The appropriate setting is that of Orlicz spaces, where we demand that the potential $\pot$ is an \nfun~and therefore convex. Hence, we obtain monotonicity of the nonlinearity $\nonl$.
We are aware of the fact that it would be desirable to weaken the monotonicity assumption and only demand the Andrews-Ball condition, although we are not yet able to prove convergence under those assumptions.

The polynomial growth and coercivity assumption leads to an $\m{Lp}$-setting where the Lebesgue space $\m{LpQ}$ over the space-time cylinder is isometrically isomorphic to the Bochner-Lebesgue space $\m{Lp0TLp}$ for $p<\infty$. This assumption allows us to reduce the partial differential equation to an operator differential equation for functions in time taking values in an appropriate Banach space of functions in space. 
However, the Orlicz space $\m{orspapotQ-introduction}$, generated by the \nfun~$\pot$, is only isometrically isomorphic to the Orlicz space $\m{orspapot0Torspapot-introduction}$ if $\pot$ is equivalent to some power function (see \cite[Proposition 1.3 on p. 218]{donaldson1974inhomogeneous}). This fact poses a main difficulty in our approach.

Our main result, which will be presented in Theorem \ref{main thm}, provides global existence of a solution. The proof shows the convergence of a subsequence of the sequence of approximate solutions, generated by a discretization in time by the backward Euler scheme and in space by a suitable Galerkin scheme. 

Qualitative studies and numerical results in the case of polynomial growth conditions have been performed by Ball, Holmes, James, Pego and Swart \cite{ball1991dynamics}, Friesecke and McLeod \cite{friesecke1996dynamics,friesecke1997dynamic}, and Carstensen and Dolzmann \cite{carstensen2004time}. Additionally, Prohl considered a finite element based full discretization of the equation
\begin{equation*}
	\partial_{tt}u- \eps\Delta \partial_{t}u - \nabla \cdot \nonl(\nabla u) = 0
\end{equation*}
for $\eps>0$, as well as for $\eps=0$ and presents numerical experiments \cite{prohl2008convergence}. The limit case $\eps =0$ constitutes the elastodynamic equation
\begin{equation}
	\partial_{tt}u -\nabla\cdot\nonl(\nabla u)=f \label{prob 2}
\end{equation}
 for which one cannot expect smooth solutions even for smooth initial data (see \cite{andrews1980existence,maccamy1967existence}). For an excellent survey of the literature concerning equation \eqref{prob 2} see \cite{emmrich2015survey}.

The remainder of this paper is structured as follows: In Section $2$, we introduce the necessary notation, give a brief introduction to Orlicz spaces and compare the growth condition \eqref{growth condition} with the restrictive $\Delta_2$-condition. The description of the numerical method we employ, the construction of the Galerkin scheme, the proof of existence and uniqueness of the numerical solution, and the derivation of a priori estimates for the fully discrete solution and the discrete time derivative follow in Section $3$. Finally, in Section $4$ we show convergence towards and, thus, existence of an exact solution, as well as its uniqueness (under additional regularity assumptions) and an error estimate for the temporal semidiscretization. The appendix contains an elementary lemma concerning the separability of the space for wich we want to construct the Galerkin scheme.

\section{Notation and Preliminaries}
After a brief survey of the notation we employ, this section provides a quick introduction to the theory of Orlicz spaces and the specific results that are needed for the rest of this paper.
\subsection{General Notation}
We keep the usual notation for function spaces. Let $\Omega\subset\R^d$ be a bounded domain. By $\m{LpOm}$ ($p\in[1,\infty]$), we denote the Lebesgue space, for $\R^d$-valued functions, we write $\m{LpOmRd}$, both equipped with the standard norm $\norm*{p,\Omega}{\cdot}$. Moreover, we rely upon the usual notation for Sobolev spaces. In particular, we have $\m{W1pOm}=\{w\in\m{LpOm}|\nabla w\in\m{LpOmRd}\}$ (with $\m{H1Om}=\m{W12Om}$), and $\m{W01pOm}$ ($p\in[1,\infty)$) denotes the closure of $\m{CcinftyOm}$ with respect to the $\m{W1p}$-norm (with $\m{H01Om}=\m{W012Om})$. Here $\m{CcinftyOm}$ denotes the space of infinitely times differentiable functions with compact support in $\Omega$. Similarly, by $\m{HrOm}$ we denote higher Sobolev spaces consisting of elements of $\m{L2Om}$ for which the derivatives up to order $r$ are again in $\m{L2Om}$.

With $\m{Lp0TX}$ ($p\in[1,\infty]$), we denote the usual Bocher-Lebesgue space, where $X$ denotes a Banach space. We recall that $\m{Lp0TLpOm}=\m{LpQ}$ if $p<\infty$. Here, we identify the abstract function $u:[0,T]\to\m{LpOm}$ with the function $u:\overline{Q}\to\R$ via $[u(t)](x)=u(x,t)$. The standard norm is then denoted by $\norm*{p,Q}{\cdot}$. The space of functions in $\m{L10TX}$ whose distributional time derivative is again in $\m{L10TX}$ is denoted by $\m{W110TX}$ and equipped with the standard norm. Analogously we define $\m{W120TX}$. By $\m{C0TX}$, $\m{AC0TX}$ and $\m{Cw0TX}$, we denote the usual spaces of uniformly continuous, absolutely continuous and demicontinuous (i.e. continuous with respect to the weak topology in $X$) functions $u:[0,T]\to X$, respectively (see also \cite{zacharias1974nichtlineare} for details). By $\langle\cdot,\cdot\rangle$, we denote the duality pairing. 
We will use the notation $X\hookrightarrow Y$ and $X\stackrel{c}{\hookrightarrow} Y$ to indicate that a Banach space $X$ is continuously respectively compactly embedded in a Banach space $Y$.
Finally, $C$ denotes a generic positive constant.

\subsection{Orlicz Spaces}

In this section, we provide the definition and basic properties of Orlicz spaces. For an introduction to Orlicz spaces, refer to \cite{kufner1977function},  as well as \cite{adams2003sobolev,gossez1974nonlinear,krasnosel1961convex,skaff1969vector,skaff1969vectorII,zeidler2013nonlinear}. Since we want to include nonlinearities with anisotropic growth, we rely upon anisotropic Orlicz classes and spaces defined by \nfun s with vector-valued arguments, as presented in \cite{skaff1969vector,skaff1969vectorII,donaldson1974inhomogeneous}.
\begin{defn}[\nfun]\label{def nfun}
	A function $\pot:\R^d\to\R$ is said to be an \nfun~if it satisfies the following conditions:
	\begin{enumerate}[label=(\roman*)]
		\item $\pot$ is continuous, $\pot(\xi)=0$ if and only if $\xi=0$, $\pot(-\xi)=\pot(\xi)$ for all $\xi\in\R^d$;\label{Nfun continuity}
		\item $\pot$ is convex;\label{Nfun convexity}
		\item $\pot$ has superlinear growth such that $\lim_{|\xi|\to\infty}\frac{\pot(\xi)}{|\xi|}=\infty$, and $\lim_{|\xi|\to 0} \frac{\pot(\xi)}{|\xi|}=0$.\label{Nfun superlinear growth}
	\end{enumerate}
\end{defn}
Some authors prefer the term generalized \nfun~in order to emphasize the dependence on $\xi$ and not only on $|\xi|$. Note that \ref{Nfun continuity} and \ref{Nfun convexity} imply $\pot(\xi)\geq 0$ for all $\xi\in\R^d$.
Because of the anisotropic character, the function $\pot$ need not be increasing with respect to the components of its vector-valued argument, e.g. $\pot(\xi_1,\xi_2)=\xi_1^2+\xi_2^2+(\xi_1-\xi_2)^2$. For more examples, refer to \cite{emmrich2013convergence}.

For an \nfun~$\pot$, $\pot^*$ denotes the conjugate function given by the Legendre-Fenchel transform $\pot^*(\eta)=\sup_{\xi\in\R^d}(\xi\cdot\eta -\pot(\xi))$, $\eta\in\R^d$. According to \cite{skaff1969vector}, the conjugate function is again an \nfun, and $\pot^{**}=\pot$. Let us recall the Fenchel-Young inequality
\begin{equation*}
	|\xi\cdot\eta|\leq \pot(\xi)+\pot^*(\xi)\quad\text{for all }\xi,\eta\in\R^d.
\end{equation*}
The \emph{anisotropic Orlicz} class $\m{orclapotOmRd}$ is the set of all (equivalence classes of almost everywhere equal) measurable functions $\xi:\Omega\to\R^d$ such that
\begin{equation*}
	\rho_{\pot,\Omega}(\xi):=\int_{\Omega}\pot(\xi(x))\dx<\infty.
\end{equation*}
Although $\m{orclapotOmRd}$ is a convex set, it may not be a linear space, e.g. if $d=1$, $\Omega=(0,1)$ and $\pot(\xi)=e^{|\xi|}-1$, then $\xi=-\frac{1}{2}\ln\in\m{orclapotOmRd}$ but $\zeta=2\xi\not\in\m{orclapotOmRd}$. The mapping $\rho_{\pot,\Omega}$ is a modular in the sense of \cite[p. 208]{kufner1977function}.

Since the function $\pot$ is continuous, $\xi=\xi(x)\in\m{LinftyOmRd}$ implies $x\mapsto\pot(\xi(x))\in\m{LinftyOm}$, which shows that $\m{LinftyOmRd}\subseteq \m{orclapotOmRd}$.

The \emph{anisotropic Orlicz space} $\m{orspapotOmRd}$ is defined as the linear hull of $\m{orclapotOmRd}$. It is a Banach space with respect to the Luxemburg norm
\begin{equation*}
	\norm{potOm}{\xi}:=\inf\left\{\lambda>0\left|\int_{\Omega}\pot\left(\frac{\xi(x)}{\lambda}\right)\dx\leq 1\right.\right\},
\end{equation*}
where the infimum is attained if $\xi\neq 0$. Let us emphasize that, in general, $\m{orspapotOmRd}$ is neither separable nor reflexive. Note that $\rho_{\pot,\Omega}(\xi)\leq \norm{potOm}{\xi}$ if $\norm{potOm}{\xi}\leq 1$ and $\rho_{\pot,\Omega}(\xi)\geq \norm{potOm}{\xi}$ if $\norm{potOm}{\xi}>1$ for all $\xi\in\m{orspapotOmRd}$. Thus
\begin{equation*}
	\norm{potOm}{\xi}\leq \rho_{\pot,\Omega}(\xi) +1.
\end{equation*}
Because of the superlinear growth of $\pot$ we have $\m{orspapotOmRd}\subseteq\m{L1OmRd}$ as shown in \cite[p. 1167]{emmrich2013convergence}.

By definition, the anisotropic Orlicz class and space coincide with the isotropic Orlicz class and space, respectively, if the \nfun~$\pot$ is a radial function.

Let us denote by $\m{orEpotOmRd}$ the closure with respect to the Luxemburg norm of the set of bounded measurable functions defined on $\Omega$. It turns out that $\m{orEpotOmRd}$ is the largest linear space contained in the Orlicz class $\m{orclapotOmRd}$ such that
\begin{equation*}
	\m{orEpotOmRd}\subseteq \m{orclapotOmRd} \subseteq \m{orspapotOmRd}
\end{equation*}
with, in general, strict inclusions. From the equivalence of the Luxemburg and the Orlicz norm
\begin{equation*}
	\norm{potOm}{\xi}^O:=\sup\left\{\left.\int_{\Omega}\xi\cdot\eta\dx\right|\eta\in\m{orclapot*OmRd}\text{ with }\rho_{\pot^*,\Omega}(\eta)\leq 1\right\},
\end{equation*}
one findes that $\m{LinftyOmRd}$ is continuously embedded in $\m{orEpotOmRd}$.

The space $\m{orEpotOmRd}$ is separable and $\m{CcinftyOmRd}$ is dense in $\m{orEpotOmRd}$. The space $\m{orspapotOmRd}$ is the dual of $\m{orEpot*OmRd}$, and the duality pairing is given by
\begin{equation*}
	\langle \xi,\eta\rangle = \int_{\Omega}\xi\cdot\eta\dx\quad \xi\in\m{orspapotOmRd},\,\eta\in\m{orEpot*OmRd}.
\end{equation*}
At this point, we may recall the generalized H\"{o}lder inequality
\begin{equation*}
	\int_{\Omega}\xi\cdot\eta\dx\leq 2\norm{potOm}{\xi}\norm{pot*Om}{\eta}\quad\text{for all }\xi\in\m{orspapotOmRd},\,\eta\in\m{orspapot*OmRd},
\end{equation*}
which shows that $\xi\cdot\eta\in\m{L1Om}$ if $\xi\in\m{orspapotOmRd}$ and $\eta\in\m{orspapot*OmRd}$. The factor $2$ in the H\"{o}lder inequality is due to the use of the Luxemburg norm instead of the Orlicz norm.

\subsection{Growth Conditions in Orlicz Spaces}
If the \nfun~$\pot$ satisfies the so-called $\Delta_2$-condition, i.e., if there exists $C>0$ such that
\begin{equation*}
	\pot(2\xi)\leq C\pot(\xi)\quad\text{for all }\xi\in\R^d,
\end{equation*}
then $\m{orEpotOmRd}=\m{orclapotOmRd}=\m{orspapotOmRd}$ (see \cite[Theorem 2.2]{skaff1969vectorII}). The $\Delta_2$-condition, however, restricts the growth significantly. For the isotropic case, it is known that the $\Delta_2$-condition is not fulfilled if the \nfun~$\pot$ grows faster than a polynomial as shown in \cite[Remark 3.4.6]{kufner1977function}.

The following proposition illustrates the connection between the $\Delta_2$-condition and other growth conditions.
\begin{prop}
	Let $\pot$ be a differentiable \nfun. Then the following two statements are equivalent:
	\begin{enumerate}[label=(\roman*)]
		\item $\pot$ satisfies the $\Delta_2$-condition.
		\item There exists a constant $C>0$ such that
	\begin{equation*}
		\pot^*(\pot'(\xi))\leq C\pot(\xi)
	\end{equation*}
	for all $\xi\in\R^d$.
	\end{enumerate}
\end{prop}

\begin{proof}
	According to \cite[Theorem 5.1]{skaff1969vector}, we have equality in the Fenchel-Young inequality if $\eta=\pot'(\xi)$. Thus $\pot^*(\pot'(\xi))=\pot'(\xi)\cdot\xi - \pot(\xi)$ and it suffices to show that the $\Delta_2$ condition is equivalent to the existence of a constant $C>0$ such that
	\begin{equation*}
		\frac{\pot'(\xi)\cdot\xi}{\pot(\xi)} \leq C+1
	\end{equation*}
	which follows from \cite[Theorem 3.2]{skaff1969vector}.
\end{proof}

\subsection{Orlicz Spaces Over the Space-Time Cylinder}
In this article, we also consider Orlicz classes and spaces over the space-time cylinder $Q$; the definitions and results introduced earlier are the same with $\Omega$ replaced by $Q$. We emphasize that $\m{orspapotQRd}\neq\m{orspapot0TorspapotOmRd}$, except for the case when $\pot$ is equivalent to a power function as proven in \cite[Proposition 1.3 on p. 218]{donaldson1974inhomogeneous}.

\section{Full Discretization}\label{sec: full discretization}
In this section, we describe the numerical approximation of \eqref{prob1}. We consider an equidistant time grid: for $N\in\N$, let $\tau=T/N$ and $t_n=n\tau$ $(n=0,1,\ldots,N)$. In addition to the time discretization, we consider an internal approximation $(V_m)_{m\in\N}$ of the space
\begin{equation*}
	V:=\{w\in\m{H01Om}|\nabla w\in\m{orEpotOmRd}\},\qquad \norm*{V}{w}:=\norm{2Om}{\nabla w} + \norm{potOm}{\nabla w}
\end{equation*}
so that $V_m\subseteq \m{W1inftyOm}$, which we will construct in the next subsection.
With respect to the right-hand side, we consider the following restriction to the time grid: for $n=1,2,\ldots,N$, let $f^n=\frac{1}{\tau}\int_{t_{n-1}}^{t_n} f(\cdot,t)\dt$.
The numerical method we consider now reads as follows: for given $u^0,v^0\in V_m$ and $f\in\m{L10TL2}$, find $(u^n)_{n=1}^N,(v^n)_{n=1}^N\subset V_m$ such that for $n=1,\ldots,N$
\begin{subequations}
\begin{align}
	\int_{\Omega}\frac{v^n-v^{n-1}}{\tau}\psi + \nabla v^n\cdot\nabla \psi +\nonl(\nabla u^n)\cdot \nabla \psi \dx &=\int_{\Omega} f^n\psi\dx\quad \text{for all }\psi\in V_m,
	\intertext{where}
	\frac{u^n-u^{n-1}}{\tau} &= v^n,
\end{align}\label{numscheme}
\end{subequations}
that is, $u^n=u_0+\tau\sum_{j=1}^n v^j$. Note that $\nonl(\nabla u^n)$ is in $\m{L1Om}$ because $u^n\in\m{W1inftyOm}$ and $\nonl$ is continuous.

The scheme \eqref{numscheme} can also be written as
\begin{equation*}
	\int_{\Omega}\frac{u^n-2u^{n-1}+u^{n-2}}{\tau^2}\psi +\frac{\nabla u^n-\nabla u^{n-1}}{\tau}\cdot\nabla \psi + \nonl(\nabla u^n)\cdot\nabla\psi = \int_{\Omega}f^n\psi
\end{equation*}
for all $\psi\in V_m$ for $n=1,2,\ldots,N$, where $u^{-1}:=u^0-\tau v^0$.

\subsection{Construction of the Galerkin Scheme}
Next, we construct a special Galerkin scheme which provides stability that we will later employ to bound the discrete second time derivative.
Let $r\in\N$ be sufficiently big such that
\begin{equation*}
	\m{Hr-1OmRd}\hookrightarrow\m{LinftyOmRd} \hookrightarrow \m{orEpotOmRd}.
\end{equation*}
Consequently, the space $\m{Hr}:=\m{HrH01}$ is densely embedded in $V$ (see Lemma \ref{Ccinfty dense in V} in the Appendix). We can then define $(\cdot,\cdot)_r$ as the canonical inner product and $\norm*{r}{\cdot}$ as the induced norm in the Hilbert space $\m{Hr}$ and let
\begin{equation*}
	T:\m{L2Om}\to\m{L2Om},\quad f\mapsto u
\end{equation*}
be the solution operator to the following problem:
\begin{equation*}
	\text{For }f\in\m{L2Om}\text{ find }u\in\m{Hr} \text{ such that }
	(u,v)_r=(f,v) \text{ for all }v\in\m{Hr}.
\end{equation*}
The operator $T$ is well-defined (Lemma of Lax-Milgram), selfadjoint, nonnegative, one-to-one and compact. Similar steps as those in \cite[Theorem 6.11 and 9.31]{brezis2010functional}
imply the existence of an orthonormal basis $(e_m)_{m\in\N}$ of $\m{L2Om}$ consisting of eigenfunctions of T, i.e.
\begin{equation*}
	Te_m = \mu_m e_m\quad\text{with }\mu_m>0,\, \mu_m\to 0\text{ for }m\to\infty.
\end{equation*}
Let $\varphi_m:=\sqrt{\mu_m}e_m$, $m\in\N$ then $(\varphi_m)_{m\in\N}$ is an orthonormal basis of $\m{Hr}$. Because of the density of the embedding of $\m{Hr}$ in $V$, the sequence $(\varphi_m)_{m\in\N}$ is a Galerkin basis of $V$ and the spaces $V_m:=\operatorname{span}\{\varphi_1,\ldots,\varphi_m\}$ form a Galerkin scheme with respect to $V$.

Now, let $P_m:\m{L2Om}\to\m{L2Om}$ denote the $\m{L2}$-orthogonal projections onto $V_m$ defined by
\begin{equation*}
	P_m v= \sum_{j=1}^m (v,e_j)e_j = \sum_{j=1}^m (v,\varphi_j)_r \varphi_j,\quad v\in\m{L2Om}.
\end{equation*}
In particular we have
\begin{equation*}
	(P_mv,v_m) = (v,v_m)\quad \text{for all }v_m\in V_m
\end{equation*}
and, since the $e_j$ are eigenfunctions of the operator $T$, $P_m$ is $\m{Hr}$-orthogonal. Therefore, we have for all $v\in\m{Hr}$
\begin{equation*}
	\norm*{r}{P_mv}^2 = (P_mv,P_mv)_r = (v,P_mv)_r\leq \norm*{r}{P_mv}\norm*{r}{v}
\end{equation*}
and thus,
\begin{equation*}
	\sup_{v\in\m{Hr}\setminus\{0\}} \frac{\norm*{r}{P_mv}}{\norm*{r}{v}} \leq 1
\end{equation*}
for all $m\in\N$.

\subsection{Existence of Approximate Solutions}
To demonstrate that the numerical scheme \eqref{numscheme} has a unique solution we use Brouwer's fixed point theorem.

\begin{thm}\label{discrete existence}
	Let $u^0,v^0\in V_m$ and $f\in\m{L10TL2}$ be given. Then there exists a unique solution $(u^n)_{n=1}^N,(v^n)_{n=1}^N$ to the numerical scheme \eqref{numscheme}.
\end{thm}

The proof of existence of solutions to the numerical scheme is based on the following auxiliary result, which is a direct consequence of Brouwer's fixed point theorem (see \cite{zacharias1974nichtlineare}).
\begin{lem}\label{cor Brouwer fp}
	For some $R>0$, let $\boldsymbol{h}:\overline{B}(0,R)\to\R^m$ be continuous, where $\overline{B}(0,R)\subset \R^m$ denotes the closed ball of radius $R$ with origin $0$ with respect to some norm $\norm*{\R^m}{\cdot}$ on $\R^m$. If
	\begin{equation*}
		\boldsymbol{h}(\boldsymbol{v})\cdot \boldsymbol{v}\geq 0 \text{ for all } \boldsymbol{v}\in\R^m\text{ with }\norm*{\R^m}{\boldsymbol{v}}=R
	\end{equation*}
	then there exists $\boldsymbol{v}^*\in\overline{B}(0,R)$ such that $\boldsymbol{h}(\boldsymbol{v}^*)=0$.
\end{lem}

\begin{proof}[Proof of theorem \ref{discrete existence}]
	We construct a one-to-one mapping between $V_m=\operatorname{span} \{\varphi_1,\ldots,\varphi_m\}$ and $\R^m$ as follows:
	\begin{align*}
		\boldsymbol{w}&=[\boldsymbol{w_1,\ldots,w_m}]\in\R^m &&\leftrightarrow& V_m\ni w&=\sum_{j=1}^{m}\boldsymbol{w_j}\varphi_j,
	\end{align*}
and $\norm*{\R^m}{\boldsymbol{w}}:=\norm{2Om}{w}$ defines a norm.
Existence and uniqueness are now shown step by step. Let us assume that $u^{n-1},u^{n-2}\in V_m$ are given. We show that there exists $u^n\in V_m$ corresponding to $\boldsymbol{u^n}\in\R^m$ being a zero of the mapping $\boldsymbol{h}=[\boldsymbol{h_1,\ldots, h_m}]:\R^m\to\R^m$ defined by
\begin{equation*}
		\boldsymbol{h_j}(\boldsymbol{w}):= \int_{\Omega} \bigg( \frac{w-2u^{n-1}+u^{n-2}}{\tau^2}\varphi_j +\frac{\nabla w-\nabla u^{n-1}}{\tau} \cdot\nabla\varphi_j
		 +\nonl(\nabla w)\cdot \nabla\varphi_j -f^n\varphi_j \bigg) \dx,
	\end{equation*}
	for $j=1,2,\ldots,m$. The continuity of $\boldsymbol{h}:\R^m\to\R^m$ is a consequence of the continuity of $\nonl$ together with the fact that $V_m\subset \m{W1inftyOm}$. With the Cauchy-Schwarz inequality and the monotonicity of $\nonl$ (note that $\nonl(0)=0$ which follows from the properties of $\pot$), we obtain
	\begin{equation*}
  \begin{split}
		\boldsymbol{h}(\boldsymbol{w})\cdot \boldsymbol{w}
		&= \int_{\Omega} \bigg( \frac{w-2u^{n-1}+u^{n-2}}{\tau^2} w +\frac{\nabla w-\nabla u^{n-1}}{\tau}\cdot\nabla w +\nonl(\nabla w)\cdot \nabla w -f^n w \bigg) \dx\\
		&\geq \frac{1}{\tau^2}\norm{2Om}{w}^2 - \frac{1}{\tau^2}\norm{2Om}{u^{n-1}}\cdot\norm{2Om}{w} -\frac{1}{\tau}\norm{2Om}{v^{n-1}} \cdot\norm{2Om}{w}  
    \\ & \hspace{8em}+ \frac{1}{\tau}\norm{2Om}{\nabla w}^2 -\frac{1}{\tau}\norm{2Om}{\nabla u^{n-1}} \cdot\norm{2Om}{\nabla w} - \norm{2Om}{f^n}\cdot\norm{2Om}{w} \\
		&= \frac{1}{\tau}\norm{2Om}{w} \left( \frac{1}{\tau}\norm{2Om}{w} - \frac{1}{\tau}\norm{2Om}{u^{n-1}} - \norm{2Om}{v^{n-1}} - \tau \norm{2Om}{f^n} \right) 
    \\ &\hspace{16em} + \frac{1}{\tau}\norm{2Om}{\nabla w} \left(\norm{2Om}{\nabla w} - \norm{2Om}{\nabla u^{n-1}}\right).
  \end{split}
	\end{equation*}
	Choosing $R=\norm{2Om}{w}$ sufficiently large and incorporating the Poincar\'e-Friedrichs inequality allows us to apply Lemma \ref{cor Brouwer fp}, providing existence of a zero of $\boldsymbol{h}$ and thus a solution to \eqref{numscheme} at level $n$.

	Let $w_1,w_2$ be two solutions of \eqref{numscheme} at level $n$. Then, in view of the monotonicity of $\nonl$ (and $\nonl(0)=0$), we have
	\begin{equation*}
	\begin{split}
		&\frac{1}{\tau^2}\norm{2Om}{w_1-w_2}^2 + \frac{1}{\tau}\norm{2Om}{\nabla w_1-\nabla w_2}^2 \\
		&= \int_{\Omega}\left(\left( \frac{w_1-2u^{n-1}+u^{n-2}}{\tau^2} - \frac{w_2-2u^{n-1}+u^{n-2}}{\tau^2} \right)(w_1-w_2) \right.\\
		&\hspace{4em}\left. + \left(\frac{\nabla w_1-\nabla u^{n-1}}{\tau}-\frac{\nabla w_2-\nabla u^{n-1}}{\tau}\right)\cdot (\nabla w_1-\nabla w_2) \right)\dx\\
		& = -\int_{\Omega} (\nonl(\nabla w_1)-\nonl(\nabla w_2))\cdot(\nabla w_1-\nabla w_2)\dx \leq 0
	\end{split}
	\end{equation*}
	which proves uniqueness.
\end{proof}

\subsection{A Priori Estimates}
The following a priori estimates are the essential prerequisite for the proof of convergence.
\begin{thm}[discrete a priori estimate I]
	The discrete solutions $(u^n)_{n=1}^N,(v^n)_{n=1}^N$ from theorem \ref{discrete existence} satisfy the following a priori estimate for $n=1,2,\ldots, N$:
	\begin{multline}
		\norm{2Om}{v^{n}}^2 +\sum_{j=1}^n \norm{2Om}{v^{j}-v^{j-1}}^2 + 2\tau\sum_{j=1}^n \norm{2Om}{\nabla v^{j}}^2 +2\int_{\Omega} \pot(\nabla u^n)\dx \\
    \leq C \left( \norm{2Om}{v^{0}}^2 +\int_{\Omega}\pot(\nabla u^0)\dx +\norm{L10TL2}{f}^2 \right).\label{dape}
	\end{multline}
\end{thm}

\begin{proof}
	We test the first equation of \eqref{numscheme} with $\psi=v^n$ and employ the convexity inequality, the Cauchy-Schwarz inequality and the identity
	\begin{equation}
		(A-B)\cdot A =\frac{1}{2}(A^2-B^2 +(A-B)^2),\label{(a-b)a rule}
	\end{equation}
	which holds true for all $A,B\in\R$ as well as $A,B\in\R^{d}$. We find
	\begin{align*}
		\frac{1}{2\tau} \left( \norm{2Om}{v^{n}}^2 - \norm{2Om}{v^{n-1}}^2 + \norm{2Om}{v^{n}-v^{n-1}}^2 \right)
		&+ \norm{2Om}{\nabla v^{n}}^2 
    +\frac{1}{\tau}\int_{\Omega}( \pot(\nabla u^n)-\pot(\nabla u^{n-1}) )\dx \\ &\leq \norm{2Om}{f^n}\norm{2Om}{v^{n}}.
	\end{align*}
	Summation then implies for all $n=1,2,\ldots,N$
	\begin{multline}
		\norm{2Om}{v^{n}}^2 + \sum_{j=1}^n \norm{2Om}{v^{j}-v^{j-1}}^2 +2\tau\sum_{j=1}^n \norm{2Om}{\nabla v^{j}}^2 +
    2\int_{\Omega} \pot(\nabla u^n)\dx\\
		\leq \norm{2Om}{v^{0}}^2 + 2\int_{\Omega} \pot(\nabla u^0)\dx +2\tau \sum_{j=1}^n \norm{2Om}{f^{j}}\norm{2Om}{v^{j}}. \label{intermediate_dape}
	\end{multline}
	Taking $n$ such that $\norm{2Om}{v^n}=\max_{j=1,2,\ldots,N}\norm{2Om}{v^j}=:X$ and using
	\begin{equation*}
		\tau\sum_{j=1}^N \norm{2Om}{f^j}\leq \norm*{\m{L10TL2}}{f},
	\end{equation*}
	results in the quadratic inequality
	\begin{equation*}
		X^2\leq \norm{2Om}{v^0}^2 + 2 \int_{\Omega}\pot(\nabla u^0)\dx + 2\norm*{\m{L10TL2}}{f}X,
	\end{equation*}
	implying
	\begin{equation*}
		X\leq \norm{2Om}{v^0} + \sqrt{2} \left(\int_{\Omega}\pot(\nabla u^0)\dx\right)^{\frac{1}{2}} + 2\norm*{\m{L10TL2}}{f}.
	\end{equation*}
	Going back to \eqref{intermediate_dape} proves the assertion.
\end{proof}

\begin{thm}[discrete a priori estimate II]
	The discrete solutions $(u^n)_{n=1}^N,(v^n)_{n=1}^N$ from Theorem \ref{discrete existence} satisfy the following a priori estimate for $n=1,2,\ldots, N$:
	\begin{equation}
		\tau\sum_{n=1}^N \norm*{(\m{Hr})^*}{\frac{v^n-v^{n-1}}{\tau}} \leq C \bigg( \norm{2Om}{v^{0}} +\int_{\Omega}\pot(\nabla u^0)\dx +\norm{L10TL2}{f} 
		+ \max_{n=1,\ldots,N}\norm*{\pot^*,\Omega}{\nonl(\nabla u^n)}\bigg).\label{dapeII}
	\end{equation}
\end{thm}

\begin{proof}
	Since $v^n$ and $v^{n-1}$ are in $V_m\subset V\subset\m{L2Om}$ and due to the $\m{Hr}$-orthogonality of the projection $P_m$, we have
	\begin{align*}
		\norm*{(\m{Hr})^*}{\frac{v^n-v^{n-1}}{\tau}} &= \sup_{v\in\m{Hr}\setminus\{0\}}\frac{1}{\norm*{r}{v}} \left( \frac{v^n-v^{n-1}}{\tau},v\right)_r\\
		& = \sup_{v\in\m{Hr}\setminus\{0\}} \frac{1}{\norm*{\m{Hr}}{v}} \frac{\norm*{\m{Hr}}{P_m v}}{\norm*{\m{Hr}}{P_m v}} \left( \frac{v^n-v^{n-1}}{\tau},P_m v\right)_{\m{L2Om}}.
	\end{align*}
	Since $(v^n)_{n=0}^N$ satisfies the first equation in equation \eqref{numscheme} and $P_mv\in V_m$, we obtain
	\begin{align*}
		\norm*{(\m{Hr})^*}{\frac{v^n-v^{n-1}}{\tau}} 
		= \sup_{v\in\m{Hr}\setminus\{0\}} \frac{\norm*{\m{Hr}}{P_m v}}{\norm*{\m{Hr}}{v}} \cdot \frac{\int_{\Omega} (f^n\cdot P_m v - \nabla v^n\cdot\nabla P_mv -\nonl(\nabla u^n)\cdot\nabla P_m v )\dx}{\norm*{\m{Hr}}{P_m v}}
	\end{align*}
	Employing the Cauchy-Schwarz inequality and the generalized H\"{o}lder inequality, we find
	\begin{align*}
    \int_{\Omega} (f^n \cdot &P_m v - \nabla v^n \cdot\nabla P_mv -\nonl(\nabla u^n)\cdot\nabla P_m v )\dx \\
    &\leq C \Big(\norm{2Om}{f^n}\cdot\norm{2Om}{P_m v} + \norm{2Om}{\nabla v^n}\cdot\norm{2Om}{\nabla P_m v}  + \norm{pot*Om}{\nonl(\nabla u^n)}\cdot\norm{potOm}{\nabla P_m v}\Big)\\
    & \leq C \norm*{V}{P_m v}\left(\norm{2Om}{f^n} + \norm{2Om}{\nabla v^n} + \norm{pot*Om}{\nonl(\nabla u^n)}\right).
  \end{align*}
  Since the continuity of the embedding $\m{Hr}\hookrightarrow V$ implies $ \frac{1}{\norm*{\m{Hr}}{P_m v}} \leq C \frac{1}{\norm*{V}{P_m v}}$ for $v\neq 0$, together with the stability of the projections $P_m$ we obtain
  \begin{align*}
		\norm*{(\m{Hr})^*}{\frac{v^n-v^{n-1}}{\tau}} &\leq C \left(\norm{2Om}{f^n}+\norm{2Om}{\nabla v^n} +\norm{pot*Om}{\nonl(\nabla u^n)} \right).
	\end{align*}
	Multiplying by $\tau$ and summing from $n=1$ to $N$ yields
	\begin{align*}
		\tau\sum_{n=1}^N\norm*{(\m{Hr})^*}{\frac{v^n-v^{n-1}}{\tau}} 
		\leq C \left(\tau\sum_{n=1}^N \norm{2Om}{f^n}+\tau\sum_{n=1}^N \norm{2Om}{\nabla v^n} +\max_{n=1,\ldots,N}\norm{pot*Om}{\nonl(\nabla u^n)} \right).
	\end{align*}
	The claim now follows from $\tau\sum_{n=1}^N \norm{2Om}{f^n}\leq \norm*{\m{L10TL2}}{f}$ and the previous a priori estimate in Theorem \ref{dape}.
\end{proof}

\section{Existence via Convergence of Approximate Solutions}
In the following, let us consider sequences $(m_l)_{l\in\N}$ and $(N_l)_{l\in\N}$ of positive integers such that $m_l,N_l\to\infty$ as $l\to\infty$. The discrete solution to \eqref{numscheme} corresponding to the discretization parameters $m_l,N_l$ (with $\tau_l:=T/N_l$) shall be denoted by $(u_l^n)_{n=0}^{N_l},(v_l^n)_{n=0}^{N_l}$, where $u_l^0\in V_{m_l}$ and $v_l^0\in V_{m_l}$ denote the approximate initial values. We do not explicitly denote the dependence of $t_n=n\tau_l$ on $l$.

Regarding the approximation of the initial values, we assume that
\begin{equation}
	u_l^0\to u_0 \text{ in }V \quad \text{and} \quad v_l^0\to v_0\text{ in }\m{L2Om} \text{ as }l\to\infty.\label{initial data}
\end{equation}
From the discrete solution, we construct approximate solutions defined on the whole time interval as follows: let $u_l$ denote the piecewise constant function such that
\begin{equation*}
	u_l(\cdot, t)= u_l^n\quad \text{if }t\in (t_{n-1},t_n]\quad(n=1,2,\ldots, N_l),\quad u_l(\cdot,0)=u_l^1,
\end{equation*}
and let $\widehat{u}_l$ denote the linear spline interpolating $(t_0,u_l^0),(t_1,u_l^1),\ldots,(t_{N_l},u_l^{N_l})$, i.e.
\begin{align*}
	\widehat{u}_l(t) &= u_l^{n-1} + \frac{u_l^n-u_l^{n-1}}{\tau_l} (t-t_{n-1})\\
  &= \frac{t_n - t}{\tau_l}u_l^{n-1} + \frac{t- t_{n-1}}{\tau_l}u_l^n \qquad\text{for }t\in[t_{n-1},t_n]\quad (n=1,2,\ldots,N_l).
\end{align*}
In an analogous way, we define $v_l$ and $\widehat{v}_l$, as well as the piecewise constant function $f_l$.

The primary result of this paper can be summarized by the following theorem.

\begin{thm}\label{main thm}
	Let $u_0\in V$, $v_0\in\m{L2Om}$ and $f\in\m{L10TL2}$. Further, let $\nonl$ satisfy the growth condition
	\begin{equation*}
		\pot^*(\nonl(A))\leq C(1+\pot(A))
	\end{equation*}
	for all $A\in\R^{d}$. Then there exists a weak solution $u\in\m{Cw0TH01}$ with $\partial_{t}u\in\m{Cw0TL2},$ $\nabla u\in\m{orclapotQRd}$ and $\nonl(\nabla u)\in\m{orclapot*QRd}$ to \eqref{prob1} in the sense of distributions, that is,
	\begin{equation*}
		\int_Q (-\partial_t u \partial_t w + \nabla\partial_t u\cdot\nabla w + \nonl(\nabla u)\cdot \nabla w)\dx\dt = \int_Q fw\dx\dt
	\end{equation*}
	for all $w\in\m{CcinftyQ}$, with $u(\cdot,0)=u_0$ in $V$, and $\partial_t u(\cdot,0)=v_0$ in $\m{L2Om}$.
\end{thm}
This solution is the limit of a subsequence, denoted by $l$ throughout this paper, of approximate solutions constructed from \eqref{numscheme} in the following sense:
The piecewise constant and piecewise linear temporal interpolation $u_l$ and $\widehat{u}_l$ converge weakly* in $\m{Linfty0TH01}$ and strongly in $\m{C0TL2}$, respectively, towards $u$. The piecewise constant temporal interpolation $v_l$ of the discrete time derivatives converges weakly* in the space $\m{Linfty0TL2}$ and weakly in $\m{L20TH01}$ towards $\partial_t u$ and the piecewise linear in time interpolation $\widehat{v}_l$ converges strongly in $\m{L2Q}$ towards $\partial_t u$. Moreover, $\nabla u_l$ converges weakly in $\m{orspapotQRd}$ towards $\nabla u$ and $\nonl(\nabla u_l)$ converges weakly* in $\m{orspapot*QRd}$ towards $\nonl(\nabla u)$.

\begin{lem}[Convergence of subsequences I]\label{conv subseq I}
	Let $u_0\in V$, $v_0\in\m{L2Om}$ and $f\in\m{L10TL2}$ and let the approximations of the initial values $(u_l^0)$ and $(v_l^0)$ satisfy \eqref{initial data}. Then there exists a subsequence, still denoted by $l$, and some $u\in\m{Cw0TH01}$ with $\partial_t u\in\m{Linfty0TL2}\cap\m{L20TH01}$ and $\nabla u\in\m{orclapotQRd}$, as well as $\xi\in\m{H01Om}$ and $\zeta\in\m{L2Om}$ such that, as $l\to\infty$,
	\begin{enumerate}[label=(\Roman*)]
	\item $u_l\wkconv*{*}{} u$ in $\m{Linfty0TH01}$, \label{ul wk* convergence}
	\item $\widehat{u}_l-u_l \to 0$ in $\m{L20TH01}$, \label{uhatl - ul strong convergence}
	\item $\widehat{u}_l\to u$ in $\m{C0TL2}$, \label{uhatl strong convergence}
	\item $v_l\wkconv*{*}{} \partial_{t}u$ in $\m{Linfty0TL2}$,\label{vl wk* convergence}
	\item $\widehat{v}_l-v_l\to 0$ in $\m{L2Q}$,\label{vhatl - vl strong convergence}
	\item $v_l\wkconv{}{} \partial_{t}u$ in $\m{L20TH01}$, \label{vl wk convergence}
	\item $\nabla u_l\wkconv*{*}{} \nabla u$ in $\m{orspapotQRd}$,\label{nabla ul wk* convergence}
	\item $\widehat{u}_l(T)\wkconv{}{}\xi$ in $\m{H01Om}$ and $\widehat{v}_l(T)\wkconv{}{}\zeta$ in $\m{L2Om}$.
\end{enumerate}
\end{lem}

\begin{rem}
	Under the assumptions of Theorem \ref{conv subseq I}, a subsequence of $(u_l)$ converges strongly in $\m{Linfty0TL2}$ towards $u$ because
	\begin{equation*}
		\norm*{\m{Linfty0TL2}}{u_l-\widehat{u}_l} \leq \tau_l\norm*{\m{Linfty0TL2}}{v_l}\to 0\text{ as }l\to\infty.
	\end{equation*}
	Moreover, if $X$ is an intermediate Banach space between $\m{L2Om}$ and $\m{H01Om}$ in the sense of Lions and Peetre \cite{lions1964classe} such that there exists $C>0$ and $\theta\in (0,1)$ with
	\begin{equation*}
		\norm*{X}{w}\leq C \norm{2Om}{\nabla w}^{\theta}\norm{2Om}{w}^{1-\theta}\quad\text{for all }w\in\m{H01Om},
	\end{equation*}
	then $(\widehat{u}_l)$ is a Cauchy sequence and thus converges strongly in $\m{C0TX}$ towards $u$. As before, $(u_l)$ converges strongly in $\m{Linfty0TX}$ towards $u$.
\end{rem}

\begin{proof}[Proof of Theorem \ref{conv subseq I}]
	In view of \eqref{initial data}, the right hand side in the a priori estimate \eqref{dape} is bounded. 
	Because $u_l^n=u_l^0+\tau_l\sum_{j=1}^n v_l^j$, we find
	\begin{align*}
		\norm{2Om}{\nabla u_l^n} &\leq \norm{2Om}{\nabla u_l^0} +\tau_l\sum_{j=1}^{n}\norm{2Om}{\nabla v_l^n} 
		\leq \norm{2Om}{\nabla u_l^0} +C \left(\tau_l\sum_{j=1}^{n}\norm{2Om}{\nabla v_l^j}^2\right)^{\frac{1}{2}}.
	\end{align*}
	Thus, as a consequence of the discrete a priori estimate \eqref{dape}, the sequences $(u_l)_{l\in\N},(\widehat{u}_l)_{l\in\N}$ are bounded in $\m{Linfty0TH01}$ and also $(v_l)_{l\in\N},(\widehat{v}_l)_{l\in\N}$ are bounded in $\m{Linfty0TL2}$. Thus, there are a subsequence, still denoted by $l$, and elements $u,\widehat{u}\in\m{Linfty0TH01},\,v,\widehat{v}\in\m{Linfty0TL2}$ such that
	\begin{align*}
		u_l&\wkconv*{*}{}u,\,\widehat{u}_l\wkconv*{*}{}\widehat{u} \text{ in }\m{Linfty0TH01},\\
		v_l&\wkconv*{*}{}v,\,\widehat{v}_l\wkconv*{*}{}\widehat{v}\text{ in }\m{Linfty0TL2}.
	\end{align*}
	In view of \eqref{dape}
	\begin{equation*}
		\norm*{\m{L20TH01}}{\widehat{u}_l-u_l}^2=\frac{\tau_l^3}{3}\sum_{n=1}^{N_l}\norm*{\m{L2Om}}{\nabla v_l^n}^2 \to 0,
	\end{equation*}
	we find that $\widehat{u}_l-u_l\to 0$ in $\m{L20TH01}$ and thus $\widehat{u}=u$. Similarly,
	\begin{equation*}
		\norm*{\m{L2Q}}{\widehat{v}_l-v_l}^2=\frac{\tau_l}{3}\sum_{n=1}^{N_l}\norm*{\m{L2Om}}{ v_l^n -v_l^{n-1}}^2 \to 0,
	\end{equation*}
	and thus $\widehat{v}=v$.
	Because by definition $v_l=\partial_t \widehat{u}_l$, we immediately find $v=\partial_t u$.

	The sequence $(\widehat{u}_l)\subset\m{C0TL2}$ is equicontinuous because $(\partial \widehat{u}_l)=(v_l)$ is bounded in $\m{Linfty0TL2}$ and $(\widehat{u}_l(t))\subset\m{H01Om}$ is bounded in $\m{H01Om}$ and hence relatively compact in $\m{L2Om}$ for every $t\in[0,T]$. An application of Arzel\`{a}-Ascoli's theorem implies strong convergence in the space $\m{C0TL2}$ of a subsequence (again still denoted by $l$), and the limit can only be the weak*-limit $u$. 
	We have seen that $u\in\m{Linfty0TH01}$ with $\partial_t u\in\m{Linfty0TL2}$ such that $u\in\m{AC0TL2}$, and in view of \cite[Lemma 8.1 on p. 297]{lions2012non} $u\in\m{Cw0TH01}$.

	Consequently, we have proved the first five statements. The sixth follows directly as
	\begin{equation*}
		\norm*{\m{L20TH01}}{v_l}^2 = \tau_l\sum_{n=1}^{N_l}\norm{2Om}{\nabla v_l^n}^2
	\end{equation*}
	and the right hand side is bounded due to the a priori estimate \eqref{dape}. Hence, we can extract a subsequence (still denoted by $l$) and $\widetilde{v}\in\m{L20TH01}$ such that $v_l\wkconv{}{}\widetilde{v}$ in $\m{L20TH01}$. As before, we can show $\widetilde{v}=\partial_t u$.

	In view of the discrete a priori estimate \eqref{dape} we observe that
	\begin{equation*}
		\int_Q \pot(\nabla u_l)\dx\dt = \tau_l\sum_{n=1}^{N_l}\int_{\Omega}\pot(\nabla u_l^n)\dx
	\end{equation*}
	is uniformly bounded. However, from the boundedness of the modular boundedness of the Luxemburg norm follows. Therefore, the sequence $(\nabla u_l)$ is bounded in $\m{orspapotQRd}$, the dual of the separable space $\m{orEpot*QRd}$. Thus, we can extract a subsequence (still denoted by $l$) such that $\nabla u_l\wkconv*{*}{}\chi$ in $\m{orspapotQRd}$ for some $\chi$. In view of the sequential lower semicontinuity of the modular in $\m{L1QRd}$ as shown in \cite[Lemma 2.2]{emmrich2013convergence}, we have $\chi\in\m{orclapotQRd}$.

	Since $\m{CcinftyOmRd}\otimes\m{Ccinfty0T}\subseteq \m{orEpot*QRd}$ we find for all functions $\Phi\in\m{CcinftyOmRd}$ and $\Psi\in\m{Ccinfty0T}$ with integration by parts and \ref{ul wk* convergence}
	\begin{align*}
		\int_Q \chi\cdot \Phi\Psi\dx\dt &= \lim_{l\to\infty}\int_Q \nabla u_l\cdot \Phi\Psi\dx\dt\\
		&= -\lim_{l\to\infty} \int_Q u_l\nabla\cdot\Phi\Psi\dx\dt = -\int_Q u\nabla\cdot\Phi\Psi\dx\dt
	\end{align*}
	and thus $\chi=\nabla u$.

	Lastly, with the same argument as in \ref{ul wk* convergence}, the sequence $(\widehat{u}_l(\cdot,T))_{l\in\N}$ (with $\widehat{u}_l(\cdot,T)=u_l^{N_l}=u_l(\cdot,T)$) is bounded in $\m{H01Om}$ and the sequence $(\widehat{v}_l(\cdot,T))_{l\in\N}$ is bounded in $\m{L2Om}$. Thus there exist $\xi\in\m{H01Om}$ and $\zeta\in\m{L2Om}$ and subsequences (still denoted by $l$) such that $\widehat{u}(T)\wkconv{}{}\xi$ and $\widehat{v}_l(T)\wkconv{}{}\zeta$ in $\m{H01Om}$ and $\m{L2Om}$, respectively.

\end{proof}

\begin{lem}[Convergence of subsequences II]\label{conv subseq II}
	Let the assumptions of Lemma \ref{conv subseq I} be satisfied and let $\nonl$ fulfill the growth condition
	\begin{equation*}
		\pot^*(\nonl(A))\leq C(1+\pot(A))\quad \text{for all }A\in\R^{d}.
	\end{equation*}
	Additionally, assume that there is a constant $C>0$ such that the approximations of the initial value $v_0$ satisfy $\tau_l\norm*{\m{H01Om}}{v_l^0}^2<C$ for all $l\in\N$. Then, there exists $\alpha\in\m{orclapot*QRd}$ and a subsequence (still denoted by $l$) such that
	\begin{enumerate}[label=(\Roman*),start=9]
		\item $\nonl(\nabla u_l)\wkconv*{*}{}\alpha$ in $\m{orspapot*QRd}$ und\label{nonl weak* convergence}
		\item $\widehat{v}_l\to \partial_{t}u$ in $\m{L2Q}$.\label{vhat strong convergence}
	\end{enumerate}
\end{lem}

Note that the additional assumption $\tau_l\norm*{\m{H01Om}}{v_l^0}^2<C$ is fulfilled by the projections of $v_0$ onto the spaces $V_{m_l}$ if we couple the time and space discretization parameters appropriately.

\begin{proof}
	Using the growth condtion, we find
	\begin{equation*}
		\int_Q \pot^*(\nonl(\nabla u_l))\dx\dt \leq C\int_Q (1+\pot(\nabla u_l))\dx\dt
	\end{equation*}
	and as seen in the previous proof the right-hand side is bounded. Thus, $(\nonl(\nabla u_l))$ is bounded in $\m{orspapot*QRd}$, the dual of the separable space $\m{orEpotQRd}$. Therefore, we can extract a subsequence (still denoted by $l$) such that $(\nonl(\nabla u_l))$ converges weakly* towards some $\alpha\in\m{orspapot*QRd}$. Again, using the weak sequential lower semicontinuity of the modular in $\m{L1QRd}$ we find $\alpha\in\m{orclapot*QRd}$.

	For the second statement, recall that the sequence $(v_l)$ is bounded in $\m{L20TH01}$. A simple calculation then shows
	\begin{equation*}
		\norm*{\m{L20TH01}}{\widehat{v}_l}^2 \leq C  \left(\tau_l \norm*{\m{H01Om}}{v_l^0}^2 + \tau_l\sum_{n=1}^{N_l} \norm*{\m{H01Om}}{v_l^n}^2 \right)
	\end{equation*}
	and because of the assumption $\tau_l\norm*{\m{H01Om}}{v_l^0}^2\leq C$ and the discrete a priori estimate \eqref{dape}, we thus find that also $(\widehat{v}_l)$ is bounded in $\m{L20TH01}$.

	Considering the time derivative $\partial_t \widehat{v}_l$ using the discrete a priori estimate \eqref{dapeII}, we find
	\begin{align*}
		\norm*{\m{L20THr*}}{\partial_t\widehat{v}_l}^2 &= \tau_l\sum_{n=1}^{N_l}\norm*{(\m{Hr})^*}{\frac{v_l^n-v_l^{n-1}}{\tau_l}}^2\\
		&\leq C \Bigg( \norm{2Om}{v^{0}} +\int_{\Omega}\pot(\nabla u^0)\dx +\norm{L10TL2}{f} 
		+ \max_{n=1,\ldots,N}\norm*{\pot^*,\Omega}{\nonl(\nabla u^n)}\Bigg).
	\end{align*}
	As seen before (recall that $\norm{pot*Om}{\cdot}\leq 1+\rho_{\pot^*,\Omega}(\cdot)$) using the growth condition we find that the right-hand side is bounded. Considering the scale of spaces
	\begin{equation}
		\m{H01Om}\stackrel{c}{\hookrightarrow} \m{L2Om}\hookrightarrow (\m{Hr})^*
	\end{equation}
	we have seen that the sequence $(\widehat{v}_l)$ is bounded in the space
	\begin{align*}
		\mathscr{Z}&:=\{w\in \m{L20TH01} |\exists \,\partial_t w\in\m{L20THr*}\}, \\
			\intertext{equipped with the norm}
			\norm*{\mathscr{Z}}{w} &:= \norm*{\m{L20TH01}}{w} + \norm*{\m{L20THr*}}{\partial_t w}.
	\end{align*}
	The generalized Lions-Aubin lemma (see \cite[Lemma 7.7]{roubivcek2013nonlinear}) implies that $\mathscr{Z}$ is compactly embedded in $\m{L20TL2}=\m{L2Q}$. Thus, there exists a subsequence (still denoted by $l$) that converges strongly. Because of lemma \eqref{conv subseq I} \ref{vhatl - vl strong convergence}, the limit can only be $\partial_t u$.
\end{proof}

\begin{proof}[Proof of theorem \ref{main thm}]
	Using the piecewise constant and piecewise linear interpolation in time, the numerical scheme \eqref{numscheme} can be rewritten as
	\begin{align}
	\int_{\Omega} \left( \partial_{t}\widehat{v}_l(\cdot,t)\psi + \nabla v_l(\cdot,t)\cdot \nabla \psi +\nonl(\nabla u_l(\cdot,t))\cdot\nabla \psi\right) \dx=\int_{\Omega} f_l(\cdot,t) \psi \dx, \label{numsch2}
	\end{align}
	for all $\psi\in V_{m_l}$, which holds almost everywhere as well as in the weak sense on $(0,T)$, such that
	\begin{align*}
	&\int_{\Omega} \left( \widehat{v}_l(\cdot,T)\psi \Psi(T)-\widehat{v}_l(\cdot,0)\psi \Psi(0) \right)\dx \\
	&+ \int_Q \left( -\widehat{v}_l \psi \Psi' + \nabla v_l \cdot \nabla \psi \Psi +\nonl(\nabla u_l) \cdot \nabla \psi \Psi\right) \dx\dt
  	=\int_Q f_l \psi \Psi\dx\dt,
	\end{align*}
	for all $\psi\in V_{m_l}$ and $\Psi\in\m{C10T}$. Taking $\psi=R_{m_l} w$ for arbitrary $w\in V$, where $R_{m_l}$ is a restriction operator such that
	\begin{equation}
		R_{m_l}w\to w\quad\text{in }V\text{ as }l\to\infty\text{ for all }w\in V \label{restriction operator}
	\end{equation}
	(see also \cite[pp. 13 ff]{temam1973}), and employing the weak and weak* convergence shown in lemma \ref{conv subseq I} and \ref{conv subseq II}, the strong convergence of $f_l$ in $\m{L10TL2}$ towards $f$ (which follows from standard arguments) and the strong convergence of $\widehat{v}_l(\cdot,0)=v_l^0$ in $\m{L2Om}$ to $v_0$, we obtain the limit equation
	\begin{equation}
		\int_{\Omega}\left( \zeta w\Psi(T)- v_0  w\Psi(0) \right) \dx 
		+\int_{Q} \left(-\partial_{t}u w\Psi' + \nabla \partial_{t}u \cdot \nabla w\Psi+ \alpha \cdot \nabla w\Psi \right) \dx\dt
  		= \int_{Q} f w\Psi\dx\dt 
  		\label{GG1} 
	\end{equation}
	for all $w\in V$ and $\Psi\in\m{C10T}$. To be precise, we have used that, as $l\to\infty$,
	\begin{align*}
		P_{m_l}w &\to w \text{ in }\m{L2Om},\\
		P_{m_l}w \Psi' &\to w\Psi'\text{ in }\m{L10TL2},\\
		P_{m_l}w\Psi &\to w\Psi\text{ in }\m{L20TH01},\\
		\nabla P_{m_l}w\Psi &\to\nabla w\Psi\text{ in }\m{orEpotQRd},\\
		P_{m_l}w\Psi &\to w\Psi\text{ in }\m{Linfty0TL2}.
	\end{align*}
	The convergences above follow from \eqref{restriction operator} and the definition of the norm in $V$. Note also that
	\begin{equation*}
		\norm{potOm}{\nabla R_{m_l}w\Psi -\nabla w \Psi}\leq \max(1,T)\norm*{\m{C0T}}{\Psi}\norm{potOm}{\nabla R_{m_l}w-\nabla w}.
	\end{equation*}
	The limit equation \eqref{GG1} shows that
	\begin{equation}
		\frac{\operatorname{d}}{\operatorname{dt}} \int_{\Omega} \partial_{t}u w\dx = \int_{\Omega} \left( f w-\nabla \partial_{t}u \cdot \nabla w - \alpha \cdot\nabla u \right)\dx\qquad \text{for all }w\in V, \label{GG2}
	\end{equation}
	in the weak sense on $(0,T)$. The right-hand side in \eqref{GG2} is in $\m{L10T}$ because $f\in\m{L10TL2}$, $u\in\m{Linfty0TH01}$ and $\alpha\in\m{orclapot*QRd}\subseteq\m{L10Torspapot*OmRd}$.
	Since we already know that $\partial_{t}u\in\m{Linfty0TL2}$, this shows that the mapping $t\mapsto \int_{\Omega} \partial_{t}u(x,t)w(x)\dx$ is absolutely continuous on $[0,T]$ for every $w\in V$.
	Because $V$ is dense in $\m{L2Om}$ and $\partial_{t}u\in\m{Linfty0TL2}$ the mapping $t\mapsto\int_{\Omega} \partial_{t}u(x,t)w(x)\dx$ is also continuous on $[0,T]$ for every $w\in\m{L2Om}$ so that $\partial_{t}u\in\m{Cw0TL2}$.
	
	For the last step of the proof, it will be crucial to use the limit equation \eqref{GG1} not only for test functions in $V\otimes \m{C10T}$, but for a more general class of test functions.
	We will use the following approximation result almost identical to \cite[Lemma 4.3]{emmrich2014equations}.
	\begin{lem}
	Let
	\begin{equation*}
		w\in\mathscr{W} := \{ w\in \m{W110TL2}\cap\m{L20TH01}| \nabla w\in\m{orclapotQRd} \}.
	\end{equation*}
	Then for any $\eps>0$, there exists a function $w_{\eps}\in V\otimes\m{C10T}$ such that
	\begin{align*}
		\norm*{\m{W110TL2}}{w_{\eps}-w} &<\eps & \norm*{\m{L20TH01}}{w_{\eps}- w} &<\eps,
	\end{align*}
	and for all $\eta\in\m{orspapot*QRd}$
	\begin{equation*}
		\left| \int_Q \eta \cdot \nabla w_{\eps} -\int_Q \eta \cdot \nabla w \dx\dt\right| <\eps.
	\end{equation*}
\end{lem}
For any $\eps>0$, and any $w\in\mathscr{W}$, there is (recalling also the continuous embedding of the space $\m{W110TL2}$ into $\m{C0TL2}$) an element $w_{\eps}\in \m{C10T}\otimes V$ such that
\begin{multline*}
	\left| \int_{\Omega}\zeta(w_{\eps}(\cdot,T)-w(\cdot,T))\dx \right| + \left| \int_{\Omega} v_0(w_{\eps}(\cdot,0)-w(\cdot,0))\dx \right|
	+ \left| \int_Q \partial_tu \partial_t(w_{\eps}-w)\dx\dt \right| \\
	+ \left| \int_Q \nabla \partial_t u\cdot \nabla(w_{\eps}-w) \dx\dt\right|
	+ \left| \int_Q \alpha\cdot \nabla(w_{\eps}-w)\dx\dt \right| + \left| \int_Q f(w_{\eps}-w)\dx\dt \right| <\eps.
\end{multline*}
Therefore,
\begin{equation}
 	\int_{\Omega} \left(\zeta w(\cdot,T)- v_0 w(\cdot,0)\right) \dx 
 	+\int_{Q} \left( -\partial_{t}u \partial_t w + \nabla \partial_{t}u \cdot \nabla w+ \alpha \cdot \nabla w \right) \dx\dt 
 	= \int_{Q} f w\dx\dt,
 	\label{GGW}
 \end{equation}
 for all $w\in\mathscr{W}$.

\paragraph{Identification of initial and final values:}
Since $\widehat{u}_l\to u$ in $\m{C0TL2}$ as $l\to\infty$ we have in particular $\widehat{u}_l(0)\to u(0)$ in $\m{L2Om}$. On the other hand, $\widehat{u}_l(0)=u_l^0\to u_0$ in $V$ as $l\to\infty$ thus $u(0)=u_0$ in $\m{L2Om}$. Similarly we can identify $u(T)$ with $\xi$.

In order to identify $\partial_{t}u(0)$ and $ \partial_{t}u(T)$ with $v_0$ and $\zeta$ respectively we test the limit equation \eqref{GG2} with $\Psi(t)=(T-t)/T$. Thus, we find for all $w\in V$
\begin{equation*}
\begin{split}
 	\frac{\operatorname{d}}{\operatorname{dt}} \left( \Psi(t)\int_{\Omega} \partial_{t}u(\cdot, t) w \dx\right)
 	&= \Psi'(t)\int_{\Omega} \partial_{t}u(\cdot,t) w\dx + \Psi(t) \frac{\operatorname{d}}{\operatorname{dt}} \int_{\Omega} \partial_{t}u(\cdot,t) w\dx\\
 	&= \Psi'(t)\int_{\Omega} \partial_{t}u(\cdot,t) w\dx \\
 	&\hspace{8em}+ \Psi(t)  \int_{\Omega} \left( f(\cdot,t) w -\nabla \partial_{t}u(\cdot,t) \cdot\nabla w -\alpha(\cdot,t) \cdot\nabla w \right) \dx.
\end{split}
\end{equation*} 
Integration over $(0,T)$ and employing \eqref{GG1} yields
\begin{equation*}
	-\int_{\Omega} \partial_{t}u(\cdot, 0) w\dx =  \int_{\Omega} \left( \zeta w\Psi(T)-v_0 w \Psi(0) \right)\dx
  = -\int_{\Omega}v_0 w\dx.
\end{equation*}
Because of the density of $V$ in $\m{L2Om}$, we find $\partial_{t}u(0)=v_0$ in $\m{L2Om}$. A similar calculation with $\Psi(t)= t/T$ shows $\partial_{t}u(T)=\zeta$. 
\paragraph{Identification of the nonlinear term:}
Let us start by taking $\psi=u_l(\cdot,t)-u_l^0\in V_{m_l}$ as the test function in \eqref{numsch2}. After integrating over $(0,T)$, we have
\begin{equation}
	\int_Q \underset{\text{\textbf{A}}}{\nonl(\nabla u_l) \cdot \nabla u_l}\dx\dt = \int_Q \left( \underset{\text{\textbf{B}}}{f_l (u_l-u_l^0)} -\underset{\text{\textbf{C}}}{\partial_{t}\widehat{v}_l (u_l-u_l^0)} \right.
	 \left. - \underset{\text{\textbf{D}}}{\nabla v_l \cdot (\nabla u_l -\nabla u_l^0)} +\underset{\text{\textbf{E}}}{\nonl(\nabla u_l)\cdot \nabla u_l^0} \right) \dx\dt.
  \label{identification equation}
\end{equation}
We examine each term \textbf{A}-\textbf{E} seperately:
\begin{enumerate}
	\item[\textbf{A}] 
	For arbitrary $\eta\in\m{LinftyQRd}$ using the monotonicity of $\nonl$ we find
	\begin{align*}
		\int_Q \nonl(\nabla u_l)\cdot \nabla u_l \dx\dt &\geq \int_Q \nonl(\nabla u_l)\cdot \nabla u_l \dx\dt 
		- \int_Q (\nonl(\nabla u_l)-\nonl(\eta)) \cdot (\nabla u_l - \eta)\dx\dt\\
		&=\int_Q \nonl(\nabla u_l) \cdot\eta \dx\dt + \int_Q \nonl(\eta) \cdot (\nabla u_l -\eta)\dx\dt.
	\end{align*}
	Note that $\nonl(\eta)$ is in $\m{orEpot*QRd}$ since $\eta\in\m{LinftyQRd}$ and $\nonl$ is continuous. With the convergence \ref{nonl weak* convergence} and \ref{nabla ul wk* convergence} seen in lemma \ref{conv subseq II} and \ref{conv subseq I} respectively we find, as $l\to\infty$,
	\begin{equation*}
		\int_Q \alpha \cdot\eta\dx\dt +\int_Q\nonl(\eta)\cdot(\nabla u-\eta) \dx\dt\leq \liminf_{l\to\infty}\int_Q\nonl(\nabla u_l)\cdot \nabla u_l\dx\dt.
	\end{equation*}
	\item[\textbf{B}] Since $u_l$ converges weakly* in $\m{Linfty0TH01}$ and because of assumption \eqref{initial data}, $u_l-u_l^0$ converges weakly* in $\m{Linfty0TL2}$. Since $f_l\to f$ in $\m{L10TL2}$ we thus obtain
	\begin{equation*}
		\int_Q f_l(u_l-u_l^0) \dx\dt \to \int_Q f(u-u_0) \dx\dt,\quad\text{as } l\to\infty.
	\end{equation*}
	\item[\textbf{C}] Summation by parts yields
	\begin{align*}
		\int_{Q} \partial_t \widehat{v}_l(u_l-u_l^0)\dx\dt &= \sum_{n=1}^{N_l} \int_{\Omega} (v_l^n-v_l^{n-1})(u_l^n-u_l^0)\dx\\
		&=\int_{\Omega}v_l^{N_l}(u_l^{N_l}-u_l^0)\dx - \tau_l\sum_{n=1}^{N_l} \int_{\Omega}v_l^nv_l^{n-1}.
	\end{align*}
	By straightforward computation we further find
	\begin{equation*}
		\tau_l\sum_{n=1}^{N_l}\int_{\Omega} v_l^n v_l^{n-1} =\int_Q(\widehat{v}_l-v_l)v_l\dx\dt + \int_Q \widehat{v}_lv_l\dx\dt.
	\end{equation*}
 	The strong convergence of $\widehat{v}_l-v_l$ in $\m{L2Q}$ towards zero, the weak convergence of $v_l$, and the strong convergence of $\widehat{v}_l$ towards $\partial_t u$ in the space $\m{L20TH01}$ and $\m{L2Q}$ respectively provides convergence of the right-hand side. Together with the weak convergence of $\widehat{v}_l(T)$ towards $\partial_t u(T)$ in $\m{L2Om}$ and the strong convergence of $\widehat{u}_l(T)$ and $u_l^0$ towards $u(T)$ (in view of lemma \ref{conv subseq I} \ref{uhatl strong convergence}) and $u_0$ in $\m{L2Om}$ and $\m{H01Om}$ respectively we find, as $l\to\infty$,
	\begin{equation*}
		\int_{Q} \partial_t \widehat{v}_l(u_l-u_l^0)\dx\dt \to \int_{\Omega}\partial_tu(T)(u(T)-u_0)\dx - \int_Q|\partial_tu|^2\dx\dt.
	\end{equation*}

	\item[\textbf{D}]
	Since
	\begin{align*}
		\int_Q \nabla v_l\cdot \nabla u_l\dx\dt = \sum_{n=1}^{N_l} \int_{\Omega}(\nabla u_l^n-\nabla u_l^{n-1})\cdot \nabla u_l^n\dx
		\geq \frac{1}{2}\left(\norm{2Om}{\nabla u_l^{N_l}}^2-\norm{2Om}{\nabla u_l^0}^2\right)
	\end{align*}
	we find
	\begin{align*}
		-\int_Q \nabla v_l(\nabla u_l-\nabla u_l^0)\dx\dt 
		\leq -\frac{1}{2}\left(\norm{2Om}{\nabla u_l^{N_l}}^2-\norm{2Om}{\nabla u_l^0}^2\right) + \int_Q \nabla v_l\cdot\nabla u_l^0\dx\dt.
	\end{align*}
	Because $u_l^0$ converges strongly towards $u_0$ in $\m{H01Om}$ and thus strongly in $\m{L20TH01}$, together with the weak convergence of $v_l$ towards $\partial_t u$ in $\m{L20TH01}$, the weak convergence of $\widehat{u}_l(T)$ towards $u(T)$ in $\m{H01Om}$ and the weak sequential lower semicontinuity of the norm, we obtain
	\begin{align*}
		&\limsup_{l\to\infty} - \int_Q \nabla v_l(\nabla u_l-\nabla u_l^0)\dx\dt \\
		&\leq -\liminf_{l\to\infty} \frac{1}{2}\norm{2Om}{\nabla u_l^{N_l}}^2 + \lim_{l\to\infty}\frac{1}{2}\norm{2Om}{\nabla u_l^0}^2 +\lim_{l\to\infty} \int_Q \nabla v_l\cdot\nabla u_l^0\dx\dt\\
		&\leq -\frac{1}{2}\norm{2Om}{\nabla u(T)}^2 +\frac{1}{2}\norm{2Om}{\nabla u(0)}^2 +\int_Q \nabla \partial_t u\cdot\nabla u_0\dx\dt\\
		&=-\frac{1}{2}\int_0^T \frac{\operatorname{d}}{\dt}\norm{2Om}{\nabla u(t)}^2\dt +\int_Q \nabla \partial_t u\cdot\nabla u_0\dx\dt\\
		&=-\int_Q \nabla \partial_t u\cdot(\nabla u -\nabla u_0)\dx\dt.
	\end{align*}

	\item[\textbf{E}] 
	Because of assumption \eqref{initial data} $\nabla u_l^0$ converges strongly towards $\nabla u_0$ in $\m{orEpotOmRd}$ and since
	\begin{equation*}
		\norm{potQ}{\nabla u_l^0-\nabla u_0} \leq \max(1,T)\norm{potOm}{\nabla u_l^0-\nabla u_0}
	\end{equation*}
	also in $\m{orEpotQRd}$.
	Finally the weak* convergence of the nonlinear term yields convergence in the last term:
	\begin{equation*}
		\int_Q \nonl(\nabla u_l)\cdot\nabla u_l^0\dx\dt \to \int_Q \alpha \cdot \nabla u_0\dx\dt,\quad \text{as }l\to\infty.
	\end{equation*}
\end{enumerate}
Combining these results, we can pass to the limit in \eqref{identification equation} and obtain
\begin{multline}
	\int_Q \alpha \cdot\eta\dx\dt +\int_Q \nonl(\eta) \cdot (\nabla u-\eta)\dx\dt 
	\leq \int_Q f (u-u_0)\dx\dt -\int_{\Omega}\partial_{t}u(T) (u(T)-u_0)\dx \\
  + \int_Q \left(|\partial_{t}u|^2 - \nabla \partial_{t}u \cdot (\nabla u-\nabla u_0) +  \alpha \cdot\nabla u_0 \right) \dx\dt.\label{Limit inequality}
\end{multline}
We have seen in lemma \ref{conv subseq I} and \ref{conv subseq II} that $u\in\m{Linfty0TH01}$ with $\partial_t u\in\m{L2Q}$ so that $u\in\m{W110TL2}\cap\m{L20TH01}$ and additionally $\nabla u\in\m{orclapotQRd}$. This means $u\in\mathscr{W}$ is an admissable test function in \eqref{GGW}. Furthermore $u_0\in V\subset \m{H01Om}$ and $\nabla u_0\in\m{orEpotOmRd}$ and thus, as seen before, also $\nabla u_0\in\m{orEpotQRd}$. Therefore also $u_0\in\mathscr{W}$ is admissable and going back to \eqref{GGW} we find
\begin{multline*}
 	\int_{\Omega} \partial_{t}u(T) (u(T)-u_0)\dx 
 	+\int_{Q} \left(-|\partial_{t}u|^2  + \nabla \partial_{t}u \cdot (\nabla u-\nabla u_0)+ \alpha \cdot (\nabla u-\nabla u_0) \right) \dx\dt\\
  = \int_{Q} f(u-u_0)\dx\dt.
 \end{multline*}
Using this together with \eqref{Limit inequality} yields
\begin{equation*}
	\int_Q \alpha \cdot \eta\dx\dt +\int_Q \nonl(\eta) \cdot (\nabla u-\eta)\dx\dt \leq \int_Q \alpha \cdot\nabla u\dx\dt
\end{equation*}
and thus
\begin{equation}
	\int_Q (\alpha - \nonl(\eta)) \cdot (\nabla u -\eta)\dx\dt \geq 0. \label{nonnegativity for minty}
\end{equation}
The remaining step is to show that $\alpha=\nonl(\nabla u)$. To this end, we use a variant of Minty's trick adapted to the case of nonreflexive Orlicz spaces (see also \cite{emmrich2014equations,emmrich2013convergence,gwiazda2010monotonicity,mustonen1999monotone}). For $k\geq 0$ we define $Q_k:=\{(x,t)\in Q| \,|\nabla u(x,t)|>k\}$ and set
\begin{equation*}
	\eta =\begin{cases}
		0 & \text{in }Q_j,\\
		\nabla u &\text{in }Q_k\setminus Q_j,\\
		\nabla u-\lambda\overline{\eta} &\text{in }Q\setminus Q_k,
	\end{cases}
\end{equation*}
where $\lambda\in(0,1)$ and $\overline{\eta}\in\m{LinftyQRd}$ and $j>k\geq 0$ are arbitrary. This choice ensures $\eta\in\m{LinftyQRd}$ and together with \eqref{nonnegativity for minty} we find
\begin{equation*}
	\int_{Q_j}\underset{\text{\textbf{A}}}{(\alpha -\nonl(0))\cdot \nabla u}\dx\dt + \lambda \int_{Q\setminus Q_k} \underset{\text{\textbf{B}}}{(\alpha -\nonl(\nabla u-\lambda \overline{\eta}))\cdot \overline{\eta}}\dx\dt \geq 0.
\end{equation*}
\begin{enumerate}
\item[\textbf{A}] Since $\alpha,\nonl(0)\in\m{orclapot*QRd}$ and $\nabla u\in\m{orclapotQRd}$ H\"{o}lder's inequality shows $(\alpha -\nonl(0))\cdot\nabla u\in\m{L1Q}$. Because $\nabla u\in\m{L1QRd}$ the measure of $Q_j$ goes to zero as $j\to\infty$ and thus
\begin{equation*}
	\int_{Q_j}(\alpha -\nonl(0)) \cdot \nabla u \dx\dt \to 0,\quad \text{as }j\to\infty.
\end{equation*}
\item[\textbf{B}] The monotonicity of $\nonl$ gives
	\begin{equation*}
		\nonl(\nabla u-\overline{\eta}) \cdot \overline{\eta} \leq \nonl(\nabla u-\lambda\overline{\eta}) \cdot \overline{\eta} \leq \nonl(\nabla u) \cdot\overline{\eta}
	\end{equation*}
	so that
	\begin{equation*}
		|\nonl(\nabla u-\lambda\overline{\eta}) \cdot \overline{\eta}| \leq \max \{|\nonl(\nabla u-\overline{\eta})\cdot\overline{\eta}|, |\nonl(\nabla u)\cdot\overline{\eta}|\} \in \m{L1QQk},
	\end{equation*}
	because $\nabla u$ is bounded on $Q\setminus Q_k$. Since $\nonl$ is continuous, we thus find with Lebesgue's theorem on dominated convergence that
	\begin{equation*}
	\int_{Q\setminus Q_k} (\alpha -\nonl(\nabla u-\lambda\overline{\eta})) \cdot \overline{\eta} \dx\dt \to \int_{Q\setminus Q_k} (\alpha -\nonl(\nabla u))\cdot\overline{\eta}\dx\dt, \quad\text{as }\lambda\to 0.
	\end{equation*}
\end{enumerate}
Thus, we obtain
\begin{equation*}
	\int_{Q\setminus Q_k} (\alpha -\nonl(\nabla u))\cdot\overline{\eta}\dx\dt \geq 0.
\end{equation*}
With the choice
\begin{equation*}
	\overline{\eta} =\begin{cases}
		-\frac{\alpha-\nonl(\nabla u)}{|\alpha -\nonl(\nabla u)|}, &\text{if }\alpha\neq \nonl(\nabla u),\\
		0,& \text{otherwise},
	\end{cases}
\end{equation*}
we obtain
\begin{equation*}
	\int_{Q\setminus Q_k}|\alpha-\nonl(\nabla u)| \dx\dt=0.
\end{equation*}
This shows that $\alpha=\nonl(\nabla u)$ almost everywhere in $Q\setminus Q_k$. Finally, because $k$ is arbitrary, the equality holds almost everywhere in $Q$.
\end{proof}

\subsection{Uniqueness}
If the solution is sufficiently regular, we also have uniqueness. Let $u$ and $v$ be two solutions to the problem with the same data $(u_0,f)$. From the proof above, we already know that
\begin{multline}
	\int_{\Omega} (\partial_t u(\cdot,T)-\partial_t v(\cdot,T))w(\cdot,T)\dx \\
	+ \int_Q (-(\partial_t u -\partial _t v)\partial_t w + (\nabla \partial_t u - \nabla\partial_t v)\cdot\nabla w + (\nonl(\nabla u)-\nonl(\nabla v))\cdot\nabla w )\dx\dt
	= 0
\label{uniqueness test eq}
\end{multline}
for all $w\in\mathscr{W}$. If we assume that $u,v\in\mathscr{W}\cap\m{W120TH1}$ and additionally $\nonl(\nabla u),\nonl(\nabla v)\in\m{orclapot*QRd}$, then by testing \eqref{uniqueness test eq} with $w=(u-v)\Psi_{\eps,\overline{t}}$, where
\begin{equation*}
	\Psi_{\eps,\overline{t}}(t)=\begin{cases}
		1 &\text{if }0\leq t\leq \overline{t}-\eps,\\
		\frac{\overline{t}-t}{\eps} &\text{if }\overline{t}-\eps <t \leq \overline{t},\\
		0 &\text{otherwise},
	\end{cases}
\end{equation*}
and using the monotonicity of $\nonl$, we find
\begin{multline*}
	0 \geq \int_0^{\overline{t}-\eps}\int_{\Omega}((\partial_t u -\partial_t v)(u-v) + (\nabla \partial_t u-\nabla\partial_t v)\cdot(\nabla u-\nabla v))\dx\dt \\
	+ \int_{\overline{t}-\eps}^{\overline{t}} \frac{\overline{t}-t}{\eps}\int_{\Omega}((\partial_t u -\partial_t v)(u-v) + (\nabla \partial_t u-\nabla\partial_t v)\cdot(\nabla u-\nabla v))\dx\dt.
\end{multline*}
Employing Lebesgue's theorem on dominated convergence, as $\eps\to 0$ the right-hand side converges and we end up with
\begin{align*}
	0 &\geq \int_0^{\overline{t}}\int_{\Omega}((\partial_t u -\partial_t v)(u-v) + (\nabla \partial_t u-\nabla\partial_t v)\cdot(\nabla u-\nabla v))\dx\dt\\
	&=\int_0^{\overline{t}}\frac{1}{2}\frac{\operatorname{d}}{\dt}\left( \norm{2Om}{u-v}^2 + \norm{2Om}{\nabla u(\cdot,t)-\nabla v(\cdot,t)}^2 \right)\dt
\end{align*}
and thus
\begin{equation*}
	\norm*{\m{H1Om}}{u(\cdot,\overline{t})-v(\cdot,\overline{t})} =0
\end{equation*}
for all $\overline{t}\in (0,T]$, which shows uniqueness.

\subsection{Error Estimate}

Although results on additional regularity of a weak solution to the problem \eqref{prob1} are not at hand, one may ask for estimates of the discretization error providing convergence rates in case the exact solution is smooth. In this section, we make a first step towards error estimates, restricting ourselfs to the temporal semidiscretization.
\begin{thm}\label{error estimate 1}
	Let $u_0,u^0\in V, v_0,v^0\in\m{L2Om}$ and $f\in\m{L10TL2}$. Let further $u$ be a solution of \eqref{prob1} with $\partial_t u,\partial_{tt}u,\partial_{ttt}u\in\m{L10TL2}$ as well as $u(\cdot,t)\in V$, $\partial_t u(\cdot,t)\in \m{L2Om}$ and $\nonl(\nabla u(\cdot,t))\in\m{orclapot*OmRd}$ for all $t\in [0,T]$. Let $f^n:=\frac{1}{\tau}\int_{t_{n-1}}^{t_n} f(\cdot,t)\dt$ and $u^n\in V$ with $\nonl(\nabla u^n)\in\m{orclapot*OmRd}$ be an approximation of $u(\cdot,t_n)$ such that for $n=1,2,\ldots,N$,
	\begin{align*}
		\frac{u^n-u^{n-1}}{\tau}&=v^n\\
		\intertext{and}
		\int_{\Omega} \left(\frac{v^n-v^{n-1}}{\tau}v + \nabla v^n\cdot \nabla v +\nonl(\nabla u^n)\cdot\nabla v\right)\dx &= \int_{\Omega}f^nv \dx
	\end{align*}
	for all $v\in V$.Then there is a constant $C>0$ such that, for $n=1,2,\ldots,N$,
	\begin{multline*}
		\norm{2Om}{u(\cdot,t_n)-u^n} \leq C \bigg( \norm{2Om}{u_0-u^0}+\norm{2Om}{v_0-v^0} +\tau\norm*{\m{L10TL2}}{\partial_{tt}u} \\
		+\tau\norm*{\m{L10TL2}}{\partial_{ttt}u} +  \norm*{\m{L10TL2}}{\overline{f}-f} \bigg)
	\end{multline*}
	where $\overline{f}$ denotes the piecewise constant in time interpolation of $f$ with respect to $(t_n)_{n=1}^N$.
\end{thm}

\begin{proof}
	Let $e^n:=u_t(\cdot,t_n)-v^n$. Then integration by parts shows
	\begin{align*}
	  \int_{\Omega}\left( \frac{e^n-e^{n-1}}{\tau} v \right.+&\left. \nabla e^n\cdot \nabla v + (\nonl(\nabla u(\cdot,t_n))-\nonl(\nabla u^n))\cdot \nabla v\right)\dx\\
		&= \int_{\Omega} \left( \frac{\partial_t u(\cdot,t_n)-\partial_t u(\cdot,t_{n-1})}{\tau}v + \nabla \partial_t u(\cdot,t_n)\cdot\nabla v + \nonl(\nabla u(\cdot,t_n))\cdot\nabla v \right)\dx \\
		 &\hspace{.3\textwidth}-\int_{\Omega} \left(\frac{v^n-v^{n-1}}{\tau}v + \nabla v^n\cdot \nabla v +\nonl(\nabla u^n)\cdot v\right)\dx\\
		&=\int_{\Omega} \left( \frac{\partial_t u(\cdot,t_n)-\partial_t u(\cdot,t_{n-1})}{\tau}v - \partial_{tt} u(\cdot,t_n)v + (f(\cdot,t_n)-f^n)v \right)\dx\\
 		&= -\frac{1}{\tau}\int_{t_{n-1}}^{t_n}\int_{\Omega} (t-t_{n-1})\partial_{ttt} u(\cdot,t)v\dx\dt + \int_{\Omega} (f(\cdot,t_n)-f^n)v \dx
	\end{align*}
	for $n=1,2,\ldots,N$ and all $v\in V$. Testing this equation with $v=e^n$ and using the relation \eqref{(a-b)a rule} as well as the monotonicity of $\nonl$ and the Cauchy-Schwarz inequality, we find
	\begin{equation*}
		\norm{2Om}{e^n}^2 - \norm{2Om}{e^{n-1}}^2 
		\leq 2\tau \int_{t_{n-1}}^{t_n} \norm{2Om}{\partial_{ttt}u(\cdot,t)}\norm{2Om}{e^n}\dt + 2\tau\norm{2Om}{f(\cdot,t_n)-f^n}\norm{2Om}{e^n}.
	\end{equation*}
	One may check by induction that $a_n^2-a_{n-1}^2 \leq 2\tau a_nb_n$ ($n=1,2,\ldots$) for $(a_n),(b_n)\subset \R^+_0, \tau>0$, implies $a_n\leq a_0 +2\tau\sum_{j=1}^n b_j$ ($n=1,2,\ldots$). Thus, together with the estimate $\tau\sum_{n=1}^N \norm{2Om}{f(\cdot,t_n)-f^n}\leq \norm*{\m{L10TL2}}{\overline{f}-f}$ we find
	\begin{equation}
		\norm{2Om}{\partial_t u(\cdot,t_n)-v^n} \\
		\leq \norm{2Om}{v_0 - v^0} + 2\tau\norm*{\m{L10TL2}}{\partial_{ttt} u} +2\norm*{\m{L10TL2}}{\overline{f}-f}.\label{error estimate first derivative}
	\end{equation}
Having found an error estimate for the time derivative we now consider the approximation error of $u$ itself. Recalling that $u^n =u^0 +\tau\sum_{j=1}^n v^j$, we find
\begin{equation}
\begin{split}
	\norm{2Om}{u(\cdot,t_n)-u^n} &\leq \norm{2Om}{u_0-u^0} + \sum_{j=1}^n \int_{t_{j-1}}^{t_j} \norm{2Om}{\partial_t u(\cdot,t)-v^j}\dt \notag\\
	&\leq \norm{2Om}{u_0-u^0} + \sum_{j=1}^n \Bigg(\int_{t_{j-1}}^{t_j} \norm{2Om}{\partial_t u(\cdot,t)-\partial_t u(\cdot, t_j)}\dt \\
	&\hspace{12em} +\int_{t_{j-1}}^{t_j} \norm{2Om}{\partial_t u(\cdot,t_j) -v^j}\dt\Bigg)
\end{split}
\label{error estimate function}
\end{equation}
For the first integral term we can estimate as follows:
\begin{align*}
	\sum_{j=1}^n \int_{t_{j-1}}^{t_j} \norm{2Om}{\partial_t u(\cdot,t)-\partial_t u(\cdot,t_j)}\dt &\leq
	\sum_{j=1}^n \int_{t_{j-1}}^{t_j}\int_t^{t_j} \norm{2Om}{\partial_{tt} u(\cdot,s)}\operatorname{ds}\dt\\
	& \leq \tau \sum_{j=1}^n \int_{t_{j-1}}^{t_j} \norm{2Om}{\partial_{tt} u(\cdot,s)}\operatorname{ds} \\
	&\leq \tau \norm*{\m{L10TL2}}{\partial_{tt}u}. 
\end{align*}
Applying the estimate \eqref{error estimate first derivative} to the second integral term in \eqref{error estimate function} proves the statement.
\end{proof}

\appendix

\section{Separability of $V$}
As in section \ref{sec: full discretization} we let $V=\{w\in\m{H01Om}|\nabla w\in\m{orEpotOmRd}\}$ with $\norm*{V}{w}=\norm{2Om}{\nabla w}+\norm{potOm}{\nabla w}$.
	\begin{lem}\label{Ccinfty dense in V}
		The set $\m{CcinftyOm}$ is dense in $V$.
	\end{lem}

The proof of this result follows \cite[Lemma 2.3]{emmrich2013convergence} very closely.
	\begin{proof}
		Let $w\in V$ and $\eps>0$. For $n\in\N$ we define
	\begin{equation*}
		[T_n(w)](x) = \begin{cases}
			w(x)&\text{if }|w(x)|\leq n\\
			n&\text{if }w(x)>n\\
			-n &\text{if }w(x)<-n
		\end{cases}
	\end{equation*}
	and $\Omega_n=\{x\in\Omega|\,|w(x)|>n\}$. Since $w$ is measurable so is $\Omega_n$ and using Chebyshev's and  Poincar\'e-Friedrichs' inequality we have $|\Omega_n|\leq \frac{C}{n^2}\norm*{\m{H01Om}}{w}^2$. An application of Lebesgue's theorem on dominated convergence then shows
	\begin{equation*}
		\norm{2Om}{\nabla T_n(w) -\nabla w}^2 = \int_{\Omega_n} |\nabla w|^2 \dx\to 0\text{ as }n\to\infty.
	\end{equation*}
	Because $\nabla w\cdot\eta\in\m{L1Om}$ for any $\eta\in\m{orspapot*OmRd}$, the absolute continuity of the integral together with the argumentation above also shows that
	\begin{equation*}
		\int_{\Omega} (\nabla T_n(w) -\nabla w) \cdot \eta \dx \to 0\text{ as }n\to\infty.
	\end{equation*}
	Thus for $n$ sufficiently large
	\begin{align*}
		\norm{2Om}{\nabla T_n(w) -\nabla w}&<\frac{\eps}{4} & &\text{and} & \left|\int_{\Omega}(\nabla T_n(w)-\nabla w)\cdot\eta\dx\right| <\frac{\eps}{4}
	\end{align*}
	for all $\eta\in\m{orspapot*OmRd}$.

	Since $\Omega$ is a Lipschitz domain $\partial \Omega$ is compact and thus there is a finite number of points $x^{(j)}\in\partial\Omega$, radii $r_j>0$ and Lipschitz continuous functions $\lambda_j:\R^{d-1}\to\R$, $j=1,\ldots, J$ such that -- up to a rigid motion if necessary --
	\begin{equation*}
	 	\Omega\cap B(x^{(j)},r_j) = \{ y=[y_1,\ldots, y_{d-1},y_d]\in B(x^{(j)},r_j)| y_d < \lambda_j(y_1,\ldots, y_{d-1}) \},
	\end{equation*}
	where $B(x^{(j)},r_j)$ denotes the open ball of radius $r_j$ with origin $x^{(j)}$.
	For sufficiently small $\delta_0>0$ we set  $\Omega_0:=\{x\in\Omega|\operatorname{dist}(x,\partial\Omega)>\delta_0\}$ and $\Omega_j:=B(x^{(j)},r_j)$ f\"ur $j=1,\ldots,J$ so that $(\Omega_j)_{j=0}^J$ is an open cover of $\overline{\Omega}$. Let $(\chi_j)_{j=1}^J$ be a smooth partition of unity for $\overline{\Omega}$ subordinate to this open cover.

	We now set $w_j:=\chi_j T_n(w)$, $j=0,\ldots,J$ then $T_n(w)=\sum_{j=0}^J w_j$, where we extend $T_n(w)$ with zero outside of $\Omega$. Note that each $w_j$ has compact support in $\Omega_j$, $j=1,\ldots,J$, respectively, since $w\in V\subset\m{H01Om}$.

	Let
	\begin{equation*}
		J_0(x):= \begin{cases}
		c_0 \exp\left( - \frac{1}{1-|x|^2}\right) &\text{if }|x|^2<1,\\
		0&\text{otherwise,}
		\end{cases}
	\end{equation*}
	where $c_0>0$ is such that $\int_{\R^d}J_0\dx=1$, and set for sufficiently small $\delta>0$ 
	\begin{equation}
		 J_{\delta}(x) = \delta^{-d} J_0(\delta^{-1}x),\quad x\in\R^d.\label{molification}
	\end{equation}
	For any locally integrable function $u$, the mollification $J_{\delta}* u$ is then a smooth function with compact support in $B(0,\delta)$. 

	\begin{enumerate}[label=(\arabic*)]
		 \item First we consider $j=0$. We observe that $w_0\in\m{H01Om}$ with $\nabla w_0 = (\nabla \chi_0)T_n(w)+\chi_0\nabla T_n(w)\in\m{orEpotOmRd}$.
		 The mollification is continuous in $\m{H01Om}$ with respect to the strong convergence and in $\m{orEpotOmRd}$ with respect to the weak convergence (see \cite[Proposition 1.2 (3)]{donaldson1974inhomogeneous}). There exists, therefore $\delta_{00}>0$ sufficiently small such that
		 \begin{equation*}
		 	\norm{2Om}{\nabla (J_{\delta_{00}}*w_0) -\nabla w_0} < \frac{\eps}{8}
		 \end{equation*}
		 and for all $\eta\in\m{orspapot*OmRd}$
		 \begin{equation*}
		 	\left| \int_{\Omega} (\nabla (J_{\delta_{00}}*w_0) - \nabla w_0)\cdot \eta \dx\right| < \frac{\eps}{8}
		 \end{equation*}
		 gilt. Here we also used that $\nabla (J_{\delta_0}*T_n(w)) = (\nabla J_{\delta_0})* T_n(w)\in\m{orEpotOmRd}$.

		\item Next we consider $j\neq 0$. Since translation is continuous in $\m{H01Om}$ with respect to the strong convergence and continuous in $\m{orEpotOmRd}$ with respect to the weak convergence (see \cite[Proposition 1.2 (2)]{donaldson1974inhomogeneous}) and since translation and derivative commute, there exists sufficiently small numbers $\delta_j>0$ such that for the translation
		
		\begin{equation*}
		  \widetilde{w}_j (x_1,\ldots,x_{d-1},x_d) := \overline{w_j} (x_1,\ldots,x_{d-1},x_d+\delta_j)
		\end{equation*}
		(where $\overline{w}_j$ is the extension of $w_j$ with zero outside of $\overline{\Omega}$)
		we have
		\begin{equation*}
		 	\norm{2Om}{\nabla \widetilde{w}_j -\nabla w_j} < \frac{\eps}{16(J+1)}
		 \end{equation*}
		 and for all $\eta\in\m{orspapot*OmRd}$
		 \begin{equation*}
		 	\left| \int_{\Omega} (\nabla \widetilde{w}_j - \nabla w_j)\cdot \eta \dx\right| < \frac{\eps}{16(J+1)}.
		 \end{equation*}
		 As before, there exists sufficiently small numbers $\widetilde{\delta}_j>0$ such that
		 \begin{equation*}
		 	\norm{2Om}{\nabla J_{\widetilde{\delta}_j}*\widetilde{w}_j -\nabla \widetilde{w}_j} < \frac{\eps}{16(J+1)}
		 \end{equation*}
		 and for all $\eta\in\m{orclapot*OmRd}$
		 \begin{equation*}
		 	\left| \int_{\Omega} (\nabla J_{\widetilde{\delta}_j}*\widetilde{w}_j - \nabla \widetilde{w}_j)\cdot \eta \dx\right| < \frac{\eps}{16(J+1)}.
		 \end{equation*}
	\end{enumerate}
	Setting $\widetilde{w}_0:=w_0$, $\widetilde{\delta}_0:=\delta_{00}$ and $w_{\eps}:= \sum_{j=0}^J J_{\widetilde{\delta}_j}*\widetilde{w}_j$ we finally find
		\begin{equation*}
		 	\norm{2Om}{\nabla w_{\eps}-\nabla w} < \frac{\eps}{2}
		 \end{equation*}
		 and for all $\eta\in\m{orspapot*OmRd}$
		 \begin{equation*}
		 	\left| \int_{\Omega} (\nabla w_{\eps}-\nabla w)\cdot \eta \dx\right| < \frac{\eps}{2}
		 \end{equation*}
	and thus
	\begin{equation*}
		\norm*{V}{w_{\eps}-w}<\eps.
	\end{equation*}
	By construction also $w_{\eps}\in\m{CcinftyOm}$.
	\end{proof}



\bibliographystyle{siam}
\bibliography{thesis.bib}

\end{document}